\documentclass[12pt]{amsart}
\usepackage[utf8]{inputenc}
\usepackage[shortlabels]{enumitem}
\usepackage[margin=3.0cm]{geometry}
\usepackage{amssymb}
\usepackage{amsthm}
\usepackage{comment}
\usepackage{latexsym}
\usepackage[all,cmtip,arrow,2cell]{xy}
\usepackage{mathtools}
\usepackage{graphicx}
\usepackage{hyperref}
\usepackage[dvipsnames]{xcolor}
\usepackage[normalem]{ulem}

\theoremstyle{plain}
\newtheorem{Theorem}{Theorem}[section]
\newtheorem{Proposition}[Theorem]{Proposition}
\newtheorem{Lemma}[Theorem]{Lemma}
\newtheorem{Corollary}[Theorem]{Corollary}

\newtheorem{Definition}[Theorem]{Definition}

\theoremstyle{definition}

\newtheorem{Remark}[Theorem]{Remark}
\newtheorem{Example}[Theorem]{Example}

\newcommand{\bbR}{\mathbb{R}}

\newcommand{\frakg}{\mathfrak{g}}

\newcommand{\frakX}{\mathfrak{X}}
\renewcommand{\phi}{\varphi}

\DeclareMathOperator{\id}{id}
\DeclareMathOperator{\Lie}{\mathcal{L}}

\DeclareMathOperator{\Empty}{{\_\!\_\,}}

\newcommand{\DA}{d_A}
\newcommand{\chD}{\check{D}}

\newcommand{\Pish}{{\Pi^\sharp}} 
\newcommand{\omfl}{{\omega_\flat}}

\newcommand{\TorA}{T_{\!A}}
\newcommand{\TorTM}{T_{TM}}
\newcommand{\TorTxM}{T_{T^* M}}

\newcommand{\scal}[2]{\langle #1,#2 \rangle} 
\newcommand{\del}{\partial}
\newcommand{\pder}[2][]{\frac{\partial#1}{\partial#2}} 

\begin{document}

\title{Hamiltonian Lie algebroids over Poisson manifolds}

\author[C.~Blohmann]{Christian Blohmann}
\address{Max-Planck-Institut f\"ur Mathematik, Vivatsgasse 7, 53111 Bonn, Germany}
\email{blohmann@mpim-bonn.mpg.de}

\author[S.~Ronchi]{Stefano Ronchi}
\address{Julius-Maximilians-Universit\"at W\"urzburg, Institute of Mathematics, Emil-Fischer-Stra{\ss}e~31, 97074 W\"urzburg, Germany, and Mathematisches Institut, Universit\"at G\"ot\-tin\-gen, Bunsenstr.~3, 37073 G\"ottingen, Germany}
\email{stefano.ronchi@uni-goettingen.de}

\author[A.~Weinstein]{Alan Weinstein}
\address{Department of Mathematics, University of California, Berkeley, CA 94720, USA, and Department of Mathematics, Stanford University, Stanford,  CA 94305, USA}
\email{alanw@math.berkeley.edu}


\date{\today}


\begin{abstract} 
We extend to Poisson manifolds the theory of hamiltonian Lie algebroids originally developed by two of the authors for presymplectic manifolds.  As in the presymplectic case, our definition, involving a vector bundle connection on the Lie algebroid, reduces to the definition of hamiltonian action for an action Lie algebroid with the trivial connection. The clean zero locus of the momentum section of a hamiltonian Lie algebroid is an invariant coisotropic submanifold, the distribution being given by the image of the anchor. We study some basic examples: bundles of Lie algebras with zero anchor and cotangent and tangent Lie algebroids.  Finally, we discuss a suggestion by Alejandro Cabrera that the conditions for a Lie algebroid $A$ to be hamiltonian may be expressed in terms of two bivector fields on $A^*$, the natural Poisson structure on the dual of a Lie algebroid and the horizontal lift by the connection of the given Poisson structure on the base.
\end{abstract} 
\maketitle


\tableofcontents

\section{Introduction}

In \cite{BlohmannWeinstein:HamLA}, two of the authors have introduced a notion of hamiltonian Lie algebroid over a presymplectic manifold. In the present paper, we will introduce the analogous concept for Poisson manifolds. As in the presymplectic case, the definition consists of three conditions, each of which generalizes a standard condition in the case where a Lie algebroid is the  action Lie algebroid $\mathfrak g\times M$ associated with an action of a Lie algebra $\mathfrak g$ on a Poisson manifold $M$.

Recall that an action $\rho: \frakg \to \frakX(M)$ of a Lie algebra $\frakg$ on a Poisson manifold $(M, \Pi)$ is called hamiltonian when the following three conditions are satisfied for all $a,b \in \frakg$:

\begin{enumerate}[(H1)]

\item $\Pi$ is invariant, $\Lie_{\rho a} \Pi = 0$,

\item there is a momentum map $\mu:M \to \frakg^*$, $\rho a = \iota_{d\langle \mu, a \rangle} \Pi$,

\item
$\mu$ is equivariant, $\langle \mu, [a,b] \rangle = \rho a \cdot \langle \mu, b\rangle$.

\end{enumerate}

To express these conditions in terms of the Lie algebroid $A = \frakg \times M$, we make the same observations as in \cite{BlohmannWeinstein:HamLA}. It is rather obvious that the action is given by the anchor $\rho: A \to TM$ and that the momentum map can be viewed as a section of the dual vector bundle $A^*$. It is less obvious that the invariance of $\Pi$ and the properties of $\mu$ must not be required for all sections of $A$, but only for the constant ones, which can be identified with the elements of the Lie algebra $\frakg$.

For a general Lie algebroid, there is no natural notion of constant sections. We therefore have to add to our structure a linear connection on $A$ and require the properties of $\rho$ and $\mu$ to hold for all sections which are horizontal, either locally if the connection is flat or pointwise if it has curvature (Section~\ref{sec:Pointwise}). While this approach is geometrically well-motivated, it  does not lead to the most useful characterization of hamiltonian Lie algebroids (Corollary~\ref{cor:H1PoissonPW} and Proposition~\ref{prop:PointHam}). Instead, we will express the definition in terms of equations to be satisfied by tensors involving the covariant derivatives under the connection and under the opposite $A$-connection, whose definition we now recall. 

When $A$ is a Lie algebroid with anchor $\rho$, the {\em opposite} $A$-connection on $TM$ is defined by
\begin{equation}
\label{eq:DcheckDef}
  \check{D}_a v 
  = [\rho a, v] + \rho( D_v a )
  \,.
\end{equation}

Recall also that a linear connection on a vector bundle $A \to M$ gives rise to a covariant derivative $D$ on sections of not only $A$, but also of its dual $A^*$ and their tensor products. Let $a$ be a section of $A$, $\mu$ a section of $A^*$, and $v$ a vector field on the base $M$. The covariant derivative satisfies the Leibniz rule
\begin{equation}
\label{eq:DPairing}
  v \cdot \langle \mu, a \rangle
  = \langle D_v \mu, a \rangle
  + \langle \mu, D_v a \rangle
  \,,
\end{equation}
which can be viewed as definition of the {\em dual} connection on $A^*$.

In particular, the dual of the opposite $A$-connection, also denoted by $
\chD$, is defined by the Leibniz rule
\begin{equation}
\label{eq:DcheckPairing}
  \rho a \cdot \langle \beta, v \rangle
  = \langle \chD_a \beta, v \rangle
  + \langle \beta, \chD_a v \rangle
  \,,
\end{equation}
where $\beta$ is a 1-form on $M$.

\begin{Remark}
We are  using the same notation $D$ for the covariant derivative of sections of $A$ and $A^*$. Analogously, we use the same letter for the action of $\chD$ on forms and multivector fields. This is consistent with the standard notation for other operations that exist on both differential forms and multivector fields, such as the Lie derivative.    
\end{Remark}

We now define the principal concepts of this paper.  We will show in Section \ref{sec:Pointwise} that these are equivalent to the pointwise definitions mentioned above.

\begin{Definition}
\label{def:HamLAPoisson}
Let $(A,\rho, [~,~])$ be a Lie algebroid over a Poisson manifold $(M,\Pi)$. Let $D$ be the covariant derivative of a linear connection on $A$.
\begin{itemize}

\item[(H1$_{\mathrm {Poi}}$)]
$A$ is \textbf{Poisson anchored} with respect to $D$ if
\begin{equation*}
  \check{D}\Pi = 0 \,.
\end{equation*}

\item[(H2$_{\mathrm {Poi}}$)]
A section $\mu \in \Gamma(A^*)$ is a \textbf{$D$-momentum section} if
\begin{equation*}
  \rho a = \iota_{\langle D\mu, a \rangle} \Pi
\end{equation*}
for all sections $a \in \Gamma(A)$.

\item[(H3$_{\mathrm {Poi}}$)]
A $D$-momentum section $\mu$ is \textbf{bracket-compatible} if
\begin{equation*}
  (\DA \mu)(a,b) = 
  \Pi\bigl( \langle D\mu, a\rangle, \langle D\mu, b \rangle \bigr)
\end{equation*}
for all sections $a,b \in \Gamma(A)$, where $\DA$ is the Lie algebroid differential, a {degree-1} differential on the graded algebra $\Gamma(\wedge^\bullet A^*)$.
\end{itemize}
A Lie algebroid together with a connection $D$ and a section $\mu$ of $A^*$ satisfying (H1$_{\mathrm {Poi}}$) and (H2$_{\mathrm {Poi}}$) is called \textbf{weakly hamiltonian}. It is called \textbf{hamiltonian} if it satisfies (H3$_{\mathrm {Poi}}$) as well.
\end{Definition}

Definition~\ref{def:HamLAPoisson} satisfies the following two basic conditions. First, if $A = \frakg \times M$ is an action Lie algebroid equipped with the trivial connection over a Poisson manifold, we recover the usual notion of a hamiltonian action. Second, if the Poisson bivector is non-degenerate, the definition is equivalent to that presented in \cite{BlohmannWeinstein:HamLA} for the presymplectic case, which can be stated in the following form:

\begin{Definition}
\label{def:HamLAPresymplectic}
Let $(A,\rho, [~,~])$ be a Lie algebroid over a presymplectic manifold $(M,\omega)$. Let $D$ be the covariant derivative of a linear connection on $A$.
\begin{itemize}

\item[(H1$_{\mathrm {Pre}}$)]
$A$ is \textbf{presymplectically anchored} with respect to $D$ if
\begin{equation*}
  \check{D}\omega = 0 \,.
\end{equation*}

\item[(H2$_{\mathrm {Pre}}$)]
A section $\mu \in \Gamma(A^*)$ is a \textbf{$D$-momentum section} if
\begin{equation*}
  \scal{D\mu}{a} = \iota_{\rho a} \omega
\end{equation*}
for all sections $a \in \Gamma(A)$.

\item[(H3$_{\mathrm {Pre}}$)]
A $D$-momentum section $\mu$ is \textbf{bracket-compatible} if
\begin{equation*}
  (\DA \mu)(a,b) = 
  - \omega(\rho a, \rho b)
\end{equation*}
for all sections $a,b \in \Gamma(A)$, where $\DA$ is the Lie algebroid differential.
\end{itemize}
A Lie algebroid together with a connection $D$ and a section $\mu$ of $A^*$ satisfying (H1$_{\mathrm {Pre}}$) and (H2$_{\mathrm {Pre}}$) is called \textbf{weakly hamiltonian}. It is called \textbf{hamiltonian} if it satisfies (H3$_{\mathrm {Pre}}$) as well.
\end{Definition}

\begin{Remark}
(H1$_{\mathrm{Pre}}$) of Definition~\ref{def:HamLAPresymplectic} is not the original form of the axiom of \cite[Def. 3.3]{BlohmannWeinstein:HamLA}, but it was shown to be equivalent to that one in \cite[Prop.~7.8]{BlohmannWeinstein:HamLA}.
\end{Remark}

\begin{Remark}
The similarity of the definitions for Poisson and presymplectic structures suggests that there should be a definition in the case of Dirac structures which subsumes them both. A notion of hamiltonian Lie algebroids over Dirac manifolds has been proposed recently \cite{Ikeda:2023hamilton}. More general still, there might be a generalization to Lie algebroids with IM-forms in the sense of \cite{BursztynCabreraOrtiz:2009}.  We have not yet looked into this possibility.
\end{Remark}

\subsection{Summary of the paper}
In Section~\ref{sec:PoissDef}, we motivate the Definition~\ref{def:HamLAPoisson} of a hamiltonian Lie algebroid, which involves a vector bundle connection on the Lie algebroid. In Proposition~\ref{prop:PoissPresympEquiv}, we show that, in the non-degenerate case, the definitions of hamiltonian Lie algebroids over Poisson and presymplectic manifolds are equivalent. Then we give a pointwise interpretation of the axioms, which states that they hold if and only if the usual conditions of a hamiltonian action hold pointwise at all points and for all sections that are horizontal at that point. As the immediate Corollary~\ref{cor:ActAlgHam}, we show that the action of a Lie algebra on a Poisson manifold is hamiltonian if and only if the action Lie algebroid is hamiltonian with respect to the trivial connection. In Proposition~\ref{prop:H3PoissonTorsion}, we give an equivalent characterization of the bracket-compatibility of momentum sections in terms of the Lie algebroid torsion, which is very useful in the rest of the paper.

In Section~\ref{sec:Examples}, we study some  basic examples: a bundle of Lie algebras over a Poisson manifold, the cotangent Lie algebroid of a Poisson manifold, and the tangent Lie algebroid of a Poisson manifold. We show that the only condition for a bundle of Lie algebras to be hamiltonian is that the momentum section must vanish on the first commutator ideal of the Lie algebras of all fibers. (This can be satisfied trivially by the zero momentum section.) For the cotangent Lie algebroid of a Poisson manifold, we show that a connection on $T^* M$ is Poisson anchored if and only if the torsion of the dual connection on $TM$ vanishes on the symplectic leaves (Corollary~\ref{cor:CotanPoissonAnchor}). If there is a 1-form $\eta$ such that $\mu = \Pish\eta$ is nowhere vanishing, then $T^* M$ is weakly hamiltonian (Theorem~\ref{thm:MomSec}). The momentum section $\mu$ is bracket-compatible if and only if it is a Liouville vector field (Proposition~\ref{prop:Liouville}). For the tangent Lie algebroid over a Poisson manifold, we show that a connection on $TM$ is Poisson anchored if and only if the dual connection on $T^*M$ has vanishing $T^*M$-torsion (Theorem~\ref{thm:TangentTorsion}). For the tangent Lie algebroid to have a momentum section, the Poisson structure must be non-degenerate. For this case, we summarize the various equivalent characterizations of (weakly) hamiltonian Lie algebroids over symplectic manifolds in Theorem~\ref{thm:WeaklyHamEquiv} and Theorem~\ref{thm:HamEquiv}. Finally, we show that the condition to be Poisson anchored is equivalent to a connection being Poisson in the usual sense \cite[Sec.~1(e), p.~65]{BayenEtAl:1978} for both the tangent Lie algebroid (Proposition~\ref{prop:PoissConn1}) and the cotangent Lie algebroid (Proposition~\ref{prop:PoissConn2}).

In Section~\ref{sec:Cabrera}, we explore an alternative notion of compatibility of a Lie algebroid $A\to M$ and a Poisson structure $\Pi$ on $M$ suggested to us by Alejandro Cabrera. In Theorem~\ref{thm:compatibilityonA*}, we show that the horizontal lift $\hat{\Pi}$ of $\Pi$ by some connection $D$ on $A^*$ commutes with the Lie algebroid Poisson structure $\Pi_A$ on $A^*$ if and only if $A$ is Poisson anchored by $D$ and satisfies an additional relation involving the $A$-torsion, the curvature, and the characteristic distribution of the Poisson structure on $M$. We then show in Theorem~\ref{thm:momentumsectionasPoissonmap} that a section $\mu$ of $A^*$ is a bracket-compatible momentum section if and only if 
the pullback $\mu^*$ on functions maps the brackets defined by the bivector field $\hat{\Pi} + \Pi_A$ to the Poisson brackets of $\Pi$. 

\subsection*{Acknowledgments}

For comments, encouragement, and advice we would like to thank Alejandro Cabrera, Lory Aintablian, Maarten Mol, Leonid Ryvkin, and Thomas Strobl, as well as the audiences over several years who heard us present preliminary versions of this work and gave invaluable feedback. Sections 2 and 3 of this work are based on the Master's thesis of S.R.~advised by C.B.~at the University of Bonn. S.R.~was funded by the Deutsche Forschungsgemeinschaft (DFG) -- 446784784.

\section{Hamiltonian Lie algebroids from symplectic to Poisson}
\label{sec:PoissDef}

\subsection{Motivation of the axioms}\label{sec:Motivation}

We recall the ``musical'' notation for 2-forms and bivector fields. The map given by inserting a vector field in a 2-form $\omega$ on $M$ will be denoted by
\begin{align*}
  \omega_\flat: TM &\longrightarrow T^* M \\
  v &\longmapsto \iota_v \omega
  \,.
\end{align*}
Analogously, given a bivector field $\Pi$ on $M$, we have the map
\begin{align*}
  \Pi^\sharp: T^*M &\longrightarrow T M \\
  \alpha &\longmapsto \iota_\alpha \Pi
  \,.
\end{align*}
We say that $\Pi$ is the inverse of $\omega$ if $\Pi^\sharp \circ \omega_\flat = \id_{TM}$, which is the case if and only if $\omega_\flat \circ \Pi^\sharp = \id_{T^* M}$.  (Note that $\omega$ has an inverse if and only if it is non-degenerate.) Evaluating both arguments of $\omega$ on the images of $\Pi^\sharp$ leads to a minus sign,
\begin{equation}
\label{eq:PBOmega=Pi}
\begin{split}
  \omega(\Pish\alpha,\Pish\beta) 
  &= \scal{\omfl\Pish\alpha}{\Pish\beta}
  = \scal{\alpha}{\Pish\beta}
  \\
  &= -\Pi(\alpha,\beta)
  \,.
\end{split}
\end{equation}
The first axiom (H1$_{\mathrm{Poi}}$) in  Definition~\ref{def:HamLAPoisson} of hamiltonian Lie algebroids over Poisson manifolds is motivated by the following fact.

\begin{Proposition}
\label{prop:DcheckOmegaPi}
Let $\omega$ be a 2-form on $M$ with an inverse bivector field $\Pi$. Let $(A,\rho)$ be an anchored vector bundle. Let $D$ be a connection on $A$ and $\chD$ be its opposite $A$-connection on $TM$. Then $\chD \Pi = 0$ if and only if $\chD\omega = 0$.
\end{Proposition}
\begin{proof}
Using Equation~\eqref{eq:DcheckPairing}, we obtain the relation
\begin{equation*}
\begin{aligned}
  (\chD_a\omega)(\Pish\alpha,\Pish\beta)
  &= \rho a \cdot \bigl( \omega(\Pish\alpha,\Pish\beta) \bigr) 
     -\omega(\chD_a\Pish\alpha,\Pish\beta)
     -\omega(\Pish\alpha,\chD_a\Pish\beta)\\
  &= -\rho a \cdot \Pi(\alpha,\beta)
     +\scal{\chD_a\Pish\alpha}{\beta}
     -\scal{\alpha}{\chD_a\Pish\beta}\\
  &= -\rho a \cdot \Pi(\alpha,\beta)
     +\rho a \cdot \scal{\Pish\alpha}{\beta}
     -\scal{\Pish\alpha}{\chD_a\beta}\\
  &{}\quad -\rho a \cdot \scal{\alpha}{\Pish\beta}
+\scal{\chD_a\alpha}{\Pish\beta}\\
  &= \rho a \cdot \Pi(\alpha,\beta)
-\Pi(\alpha, \chD_a\beta)
-\Pi(\chD_a\alpha, \beta)\\
  &=(\chD_a\Pi)(\alpha,\beta), 
\end{aligned}
\end{equation*}
for all $a \in \Gamma(A)$ and $\alpha,\beta \in \Omega^1(M)$. If $\chD\omega = 0$, the left side vanishes, which implies that $\chD \Pi = 0$. Replacing $\alpha$ with $\omega_\flat v$ and $\beta$ with $\omega_\flat w$, we obtain the relation
\begin{equation*}
  (\chD_a\omega)(v,w)
  = (\chD_a\Pi)(\omega_\flat v, \omega_\flat w)
  \,,
\end{equation*}
for all vector fields $v,w \in \frakX(M)$. This shows that $\chD\Pi = 0$ if and only if $\chD\omega = 0$.
\end{proof}

From Proposition~\ref{prop:DcheckOmegaPi}, we conclude that, if $\omega$ is non-degenerate, then the axiom (H1$_\mathrm{Poi}$) of Definition~\ref{def:HamLAPoisson} for the Poisson case is equivalent to the axiom (H1$_\mathrm{Pre}$) of Definition~\ref{def:HamLAPresymplectic} for the presymplectic case.

For axiom (H2), the situation is even more obvious. If $\omega$ is non-degenerate, we can apply the inverse $\Pi^\sharp$ of $\omega_\flat$ to
\begin{equation*}
  \scal{D\mu}{a} 
  = \omfl(\rho a),
\end{equation*}
which yields the equation
\begin{equation}
\label{eq:PiMom1}
  \Pish\scal{D\mu}{a}
  = \rho a
  \,.
\end{equation}
We conclude that, if $\omega$ is non-degenerate, the axiom (H2$_{\mathrm{Poi}}$) of Definition~\ref{def:HamLAPoisson} for the Poisson case is equivalent to the axiom (H2$_{\mathrm{Pre}}$) of Definition~\ref{def:HamLAPresymplectic} for the presymplectic case.

The third axiom (H3$_{\mathrm{Pre}}$) of Definition~\ref{def:HamLAPresymplectic} for the presymplectic case is given by the equation
\begin{equation*}
  (\DA\mu)(a,b) = - \omega(\rho a, \rho b)
  \,.
\end{equation*}
If $\omega$ is non-degenerate and if $\mu$ satisfies~\eqref{eq:PiMom1}, we can replace the right side with
\begin{equation*}
\begin{split}
  -\omega(\rho a, \rho b) 
  &= -\omega\bigl(\Pish\scal{D\mu}{a}, \Pish\scal{D\mu}{b} \bigr)
  \\
  &= \Pi\bigl( \scal{D\mu}{a}, \scal{D\mu}{b} \bigl)
  \,,
\end{split}
\end{equation*}
where the change of sign comes from~\eqref{eq:PBOmega=Pi}. Combining the last two equations, we obtain the equation of axiom (H3$_{\mathrm{Poi}}$) of Definition~\ref{def:HamLAPoisson} in the Poisson case.

We conclude that, if $\omega$ is non-degenerate and if axiom (H2$_{\mathrm{Poi}}$) of Definition~\ref{def:HamLAPoisson} is satisfied, then the axiom (H3$_{\mathrm{Poi}}$) is equivalent to the axiom (H3$_{\mathrm{Pre}}$) of Definition~\ref{def:HamLAPresymplectic} for the presymplectic case. We can summarize this with the following proposition.

\begin{Proposition}
\label{prop:PoissPresympEquiv}
A Lie algebroid $A$ with connection D and section $\mu$ of $A^*$ over a non-degenerate Poisson manifold $(M,\Pi)$ is (weakly) hamiltonian if and only if it is (weakly) hamiltonian over the symplectic manifold $(M, \omega = \Pi^{-1})$.
\end{Proposition}

\subsection{Pointwise interpretation}
\label{sec:Pointwise}

In Propositions~4.12 and 4.13 of \cite{BlohmannWeinstein:HamLA} it was shown for the presymplectic case that the axioms of a hamiltonian Lie algebroid hold if and only if the usual conditions of a hamiltonian action hold for all horizontal sections, locally or pointwise. We will show that this is still true in the Poisson case. First, we observe that the invariance of any tensor under vector fields given by the anchor of horizontal sections of a Lie algebroid can be expressed succinctly by the opposite connection (cf.~\cite[Sec.~4.2]{BlohmannWeinstein:HamLA} and~\cite{KotovStrobl:2016}).

\begin{Lemma}
Let $\rho: A \to TM$ be an anchored vector bundle with a connection $D$. Let $\phi \in \Gamma\bigl( (TM)^{\otimes p} \otimes (T^* M)^{\otimes q} \bigr)$. The following are equivalent:
\begin{itemize}
    
\item[(i)] $\chD \phi = 0$

\item[(ii)] $\Lie_{\rho a} \phi \bigr|_m = 0$ at all $m \in M$ and for all $a \in \Gamma(A)$ that are horizontal at $m$.

\end{itemize}
\end{Lemma}
\begin{proof}
Let $v \in \frakX(M)$ be a vector field. Let $\{a_i\}$ be a local frame of $\Gamma(A)$ such that all sections $a_i$ are horizontal at $m \in M$, i.e.~$(D a_i )_m = 0$. Every local section of $A$ can be written as $a = f^i a_i$  for smooth functions $f^i \in C^\infty(M)$. By the definition~\eqref{eq:DcheckDef} of the opposite connection, we have
\begin{equation*}
\begin{split}
  \bigl( \chD_a v \bigr)_m
  &= \bigl( f^i \chD_{a_i} \bigr)_m 
  = f^i(m) \bigl( \Lie_{\rho a_i} v + \rho(D_v a) \bigr)_m
  \\
  &= f^i(m) \bigl( \Lie_{\rho a_i} v \bigr)_m
  \,.
\end{split}
\end{equation*}
We conclude that $(\chD_a v)_m = 0$ for all sections $a \in \Gamma(A)$ if and only if $\Lie_{\rho b} v \bigr|_m = 0$ for all $b \in \Gamma(A)$ that are horizontal at $m$. For a 1-form $\alpha \in \Omega^1(M)$ we use~\eqref{eq:DcheckPairing} to obtain
\begin{equation*}
\begin{split}
  \langle \chD_a \alpha, v \rangle \bigr|_m
  &= f^i(m)\bigl( 
  \rho a_i \cdot \langle \alpha, v \rangle
  - \langle \alpha, \Lie_{\rho a_i} v \rangle \bigr)_m
  \\
  &= f^i(m) \langle \Lie_{\rho a_i} \alpha, v \rangle \bigr|_m
  \,,
\end{split}
\end{equation*}
for all $v \in \frakX(M)$. As before, we conclude that $(\chD_a \alpha)_m = 0$ for all sections $a \in \Gamma(A)$ if and only if $\Lie_{\rho b} \alpha \bigr|_m = 0$ for all $b \in \Gamma(A)$ that are horizontal at $m$. Since both $\chD_a$ and $\Lie_{\rho a}$ act as derivations on tensor products of vector fields and forms, the proposition follows.
\end{proof}

\begin{Corollary}
\label{cor:H1PoissonPW}
An anchored vector bundle $\rho: A \to TM$ over a Poisson manifold $(M,\Pi)$ is Poisson anchored with respect to a connection $D$ if and only if
\begin{equation}
\label{eq:H1PoissonPW}
  \Lie_{\rho a}\Pi \,\bigr|_m = 0
\end{equation}
at all points $m\in M$ and for all sections $a$ of $A$ that are horizontal at $m$.
\end{Corollary}

\begin{Proposition}
\label{prop:PointHam}
Let $A$ be a Lie algebroid over a Poisson manifold $(M,\Pi)$. Then $\mu \in \Gamma(M,A^*)$ is a $D$-momentum section if and only if
\begin{equation*}
  (\rho a)_m
  = \iota_{d\langle\mu, a\rangle}\Pi \,\bigr|_m
\end{equation*}
for all $m \in M$ and all $a \in \Gamma(M,A)$ that are horizontal at $m$. The momentum section is bracket-compatible if and only if
\begin{equation*}
 \langle\mu, [a,b] \rangle \bigr|_m
 = \rho a \cdot \langle \mu, b \rangle \,\bigr|_m 
\end{equation*}
for all $m \in M$ and all sections $a$, $b$ of $A$ that are horizontal at $m$.
\end{Proposition}
\begin{proof}
Let $m \in M$, and let $\{a_i\}$ be a local frame of $\Gamma(A)$ such that all sections $a_i$ are horizontal at $m \in M$. Every local section of $A$ can be written as $a = f^i a_i$ for smooth functions $f^i \in C^\infty(M)$. The condition  (H2$_{\mathrm{Poi}}$) of Definition~\ref{def:HamLAPoisson} evaluated at $m$ is
\begin{equation*}
\begin{split}
  f^i(m) (\rho a_i)_m 
  &= \Pi^\sharp\bigl(  d\langle\mu, f^i a_i \rangle 
     - \langle\mu, D(f^i a_i) \rangle \bigr)_m
  \\
  &= \Pi^\sharp\bigl(  df^i \langle\mu, a_i \rangle 
     + f^i d\langle\mu, a_i \rangle
     - df^i \langle\mu, D a_i \rangle
     - f^i \langle\mu, D a_i \rangle \bigr)_m
  \\
  &= f^i(m) \Pi^\sharp\bigl( d\langle\mu, a_i \rangle \bigr)_m
  \,.
\end{split}
\end{equation*}
Since the $f^i(m)$ can be chosen arbitrarily, this equation is satisfied for all $f^i$ if and only if
\begin{equation*}
  (\rho a_i)_m 
  = \Pi^\sharp\bigl( d\langle\mu, a_i \rangle \bigr)_m
  \,,
\end{equation*}
which proves the first statement of the proposition. Assuming (H2$_\mathrm{Poi}$), the bracket compatibility of axiom (H3$_\mathrm{Poi}$) is equivalent to the vanishing of 
\begin{equation*}
\begin{split}
  (\DA\mu)(a,b) - \Pi(\scal{D\mu}{a}, \scal{D\mu}{b})
  &=
  (\DA\mu)(a,b) + \iota_{\rho b} \scal{D\mu}{a}
  \\
  &=
  (\DA\mu)(a,b) + \iota_{\rho b} \bigl(
  d \scal{\mu}{a} - \scal{\mu}{Da} \bigr)
  \\
  &=
  \rho a \cdot \scal{\mu}{b} - \langle\mu, [a,b] \rangle 
  - \scal{\mu}{D_{\rho b}a}
 \,,
\end{split}
\end{equation*}
where we have used~\eqref{eq:DPairing} and the definition of the Lie algebroid differential. The right side vanishes at $m$ if and only if
\begin{equation*}
  f^i(m) \bigl(
  \rho a_i \cdot \scal{\mu}{b} - \langle\mu, [a,b] \rangle 
  \bigr)_m
  = 0
  \,.
\end{equation*}
Since the $f^i(m)$ can be chosen arbitrarily, this proves the second statement.
\end{proof}

\begin{Corollary}
\label{cor:ActAlgHam}
The action of a Lie algebra on a Poisson manifold is (weakly) hamiltonian if and only if the action Lie algebroid is (weakly) hamiltonian with respect to the trivial connection.
\end{Corollary}
\begin{proof}
Let $\rho: \frakg \to \frakX(M)$ be the action of a Lie algebra on a Poisson manifold $(M, \Pi)$. The Lie algebra can be identified with the constant sections of the action Lie algebroid $A = \frakg \times M \to M$, which are the sections that are horizontal with respect to the trivial connection $D$. If $a$ is a constant section, then $\chD_a \Pi =  \Lie_{\rho a} \Pi$ and $D\langle \mu, a \rangle = d\langle\mu, a \rangle$. The proof follows from Corollary~\ref{cor:H1PoissonPW} and Proposition~\ref{prop:PointHam}.
\end{proof}

\begin{Remark}
Note that the trivial connection is an essential part of the data.  In fact, even if an action is not weakly hamiltonian, its action Lie algebroid could still be weakly hamiltonian with respect to a different connection. An example of this situation on a symplectic manifold was given in \cite[Example~4.4]{BlohmannWeinstein:HamLA}.
\end{Remark}

\begin{Example}
\label{ex:LiePoisson}
The cotangent Lie algebroid $T^* \frakg^* \cong \frakg \times \frakg^*$ of a Lie-Poisson manifold $\frakg^*$ is isomorphic to the action Lie algebroid of the coadjoint action. The coadjoint action is hamiltonian with momentum map the projection  $\frakg \times \frakg^* \to \frakg^*$. It follows from Corollary~\ref{cor:ActAlgHam} that $T^* \frakg^*$ is hamiltonian with respect to the trivial connection.
\end{Example}

\subsection{Bracket-compatibility in terms of Lie algebroid torsion}

In the presymplectic case, the condition of bracket-compatibility has a useful equivalent description in terms of the $A$-torsion of $D$ \cite[Prop.~5.1]{BlohmannWeinstein:HamLA}, which is defined by 
\begin{equation}
\label{eq:Atorsion}
  \TorA(a,b) := D_{\rho a} b - D_{\rho b} a - [a,b]
  \,.
\end{equation}
An analogous result holds in the Poisson case:

\begin{Proposition}
\label{prop:H3PoissonTorsion}
A $D$-momentum section $\mu$ of a Lie algebroid $A$ over a Poisson manifold $(M,\Pi)$ is bracket-compatible if and only if 
\begin{equation}
\label{eq:H3PoissonTorsion}
  \bigl\langle\mu, \TorA(a,b) \bigr\rangle 
  = - \Pi\bigl( \scal{D\mu}{a}, \scal{D\mu}{b} \bigr)
\end{equation}
for all sections $a$ and $b$ of $A$.
\end{Proposition}
\begin{proof}
Using the assumption that $\mu$ is a momentum section,  we obtain
\begin{equation}
\label{eq:H3Torsion1}
\begin{split}
  \Pi\bigl(\scal{D\mu}{a},\scal{D\mu}{b}\bigr) 
  &= \tfrac{1}{2} (
    \iota_{\scal{D\mu}{b}} \iota_{\scal{D\mu}{a}} \Pi
  - \iota_{\scal{D\mu}{a}} \iota_{\scal{D\mu}{b}} \Pi )
  \\
  &= \tfrac{1}{2} (
    \iota_{\scal{D\mu}{b}} \rho a 
  - \iota_{\scal{D\mu}{a}} \rho b )
  \\
  &= \tfrac{1}{2} (
     \scal{D_{\rho a} \mu}{b} 
   - \scal{D_{\rho b} \mu}{a} )
   \,.
\end{split}
\end{equation}
The Lie algebroid differential of $\mu$ can be expressed in terms of the Lie algebroid torsion of $D$ as
\begin{equation*}
\begin{split}
  (\DA\mu)(a,b) 
  &= \rho a \cdot \scal{\mu}{b}
   - \rho b \cdot \scal{\mu}{a}
   - \scal{\mu}{[a,b]_A}
  \\
  &= \scal{D_{\rho a}\mu}{b} + \scal{\mu}{D_{\rho a}b}
   - \scal{D_{\rho b}\mu}{a} - \scal{\mu}{D_{\rho b}a}
   - \scal{\mu}{[a,b]_A}
  \\  
  &= \scal{D_{\rho a}\mu}{b} - \scal{D_{\rho b}\mu}{a} 
     + \scal{\mu}{\TorA(a,b)}
  \\  
  &= 2\Pi\bigl( \scal{D\mu}{a},\scal{D\mu}{b} \bigr)
     + \scal{\mu}{\TorA(a,b)}
  \,,
\end{split}
\end{equation*}
where in the last step we have used~\eqref{eq:H3Torsion1}. By subtracting $\Pi(\scal{D\mu}{a},\scal{D\mu}{b})$ on both sides, we obtain
\begin{equation*}
  (\DA\mu)(a,b) 
  - \Pi\bigl( \scal{D\mu}{a}, \scal{D\mu}{b} \bigr) 
  = \Pi\bigl( \scal{D\mu}{a}, \scal{D\mu}{b} \bigr) 
    + \scal{\mu}{\TorA(a,b)}
  \,.
\end{equation*}
The left side vanishes if and only if the momentum section $\mu$ is bracket-compatible, and the right side does if and only if \eqref{eq:H3PoissonTorsion} holds.
\end{proof}

\subsection{The zero locus of the momentum section}

Let $A \to M$ be a hamiltonian Lie algebroid over the Poisson manifold $(M,\Pi)$ with momentum section $\mu: M \to A^*$. Let
\begin{equation*}
  I := \{ \langle\mu, a\rangle \in C^\infty(M) ~|~ a \in \Gamma(M,A)\}   
\end{equation*}
be the space of functions we obtain by pairing $\mu$ with all sections of the Lie algebroid. Since $f \langle\mu, a\rangle = \langle\mu, fa\rangle$ for all $f \in C^\infty(M)$, $I$ is an ideal. The zero locus $Z = \mu^{-1}(0)$ is the set of common zeros of the elements of $I$. The following statement is analogous to the presymplectic case \cite[Prop.~5.2]{BlohmannWeinstein:HamLA}.

\begin{Proposition}
\label{Zeroisotropic}
In a hamiltonian Lie algebroid over a Poisson manifold, the zero locus $Z$ of the momentum section is invariant in the sense that every orbit which meets $Z$ is contained in $Z$.  
\end{Proposition}

\begin{proof}
In the proof of Proposition~\ref{prop:PointHam}, we have shown for a hamiltonian Lie algebroid that the Lie derivative of a function $\langle\mu, b \rangle$ with respect to $\rho a$ can be expressed as
\begin{equation}
\label{eq:ZeroLocus1}
  \Lie_{\rho a} \langle\mu,b \rangle
  = \langle\mu, [a,b] + D_{\rho b } a \rangle \in I \,.
\end{equation}
This shows that $I$ is invariant under the Lie derivative of every vector field in the image of the anchor. It follows that $I$ and hence its set $Z$ of common zeros are invariant under the flow of every vector field in the image of the anchor. We conclude that every orbit of $A$ that meets $Z$ is contained in $Z$.
\end{proof}

\begin{Remark}
In the analogous proposition for the presymplectic case, it was shown that all orbits in the zero locus are isotropic \cite[Prop.~5.2]{BlohmannWeinstein:HamLA}. In the Poisson case, it is straightforward to prove the isotropy of orbits in the \emph{clean} locus of $Z$, as we will see in Theorem~\ref{thm:CoisoPoiss}. However, we do not know any example of a hamiltonian Lie algebroid over a Poisson manifold with a non-isotropic orbit in the zero locus of the momentum section. 
\end{Remark}

The zero locus of a momentum section is not necessarily a smooth submanifold. For the coisotropic property of $Z$ and the isotropic property of the orbits in $Z$, we have to restrict our attention to the \textbf{clean zero locus} $Z_{\mathrm{cl}}$, which consists of the points of smoothness where the tangent space of the zero locus is the entire zero space of the differential of the momentum section. The clean zero locus can be identified in algebraic terms as the set of points $m \in M$ for which there is a neighborhood $U$ on which $Z$ is a smooth submanifold and on which the defining ideal $I$ is no smaller than the ideal $I_Z\supseteq I$ consisting of \emph{all} functions vanishing on $Z$. 


\begin{Theorem}
\label{thm:CoisoPoiss}
The clean zero locus $Z_{\mathrm{cl}}$ of the momentum section for a hamiltonian Lie algebroid over a Poisson manifold is a coisotropic submanifold which is invariant under the Lie algebroid. The characteristic distribution of $Z_{\mathrm{cl}}$ is equal to the image of the anchor. The orbits in $Z_\mathrm{cl}$ are isotropic.
\end{Theorem}
\begin{proof}
Since $Z$ is determined by $I$, so is $I_Z$, and hence the latter is invariant under all the diffeomorphisms generated by the image of $\rho$.  It follows that the subset where they agree locally is invariant.  So is the set of smooth points of $Z$, and hence so is $Z_{\mathrm{cl}}$.

Now let $m$ belong to the clean zero locus. A vector $v \in T_m M$ is tangent to $\mu^{-1}(0)$ if and only if for all $a \in \Gamma(M,A)$,
\begin{equation*}
\begin{split}
  0 
  &= v \cdot \langle\mu, a \rangle
  = \iota_v \bigl( d\langle \mu, a \rangle \bigr)_m
  = \iota_v (\langle D\mu, a \rangle_m + \langle \mu, Da \rangle_m)
  \\
  &= \iota_v \langle D\mu, a \rangle
  \,,
\end{split}
\end{equation*}
where in the last step we have used that $\mu$ vanishes at $m$. This shows that the annihilator $(T_m Z_\mathrm{cl})^\circ$ or, in other words, the fiber at $m$ of the conormal bundle of $Z_\mathrm{cl}$, is spanned by the 1-forms $\langle D\mu, a \rangle_m$ for all $a$.

Since $\mu$ is a momentum section, we have $\Pish(\langle D\mu, a \rangle) = \rho a$. It follows that
\begin{equation}
\label{eq:ZeroLocus2}
  \Pish\bigl( (T_m Z_\mathrm{cl})^\circ \bigr)
  = \rho(A_m) \subset TZ_\mathrm{cl}
  \,,
\end{equation}
where the inclusion on the right follows from the invariance of $Z_\mathrm{cl}$. Equation~\eqref{eq:ZeroLocus2} shows that $Z_\mathrm{cl}$ is coisotropic and that the characteristic distribution at $m$ is $\rho(A_m)$.

The tangent space of an orbit $S$ through $m \in Z_\mathrm{cl}$ is given by $TS = \rho(A_m) \subset T_m Z_\mathrm{cl}$. It follows that $(T_m S)^\circ \supset (T_m Z_\mathrm{cl})^\circ$ and, therefore, $\Pish\bigl( (T_m S)^\circ\bigr) \supset \Pish\bigl( (T_m Z_\mathrm{cl})^\circ \bigr)$. With Equation~\eqref{eq:ZeroLocus2}, we conclude that
\begin{equation*}
  \Pish\bigl( (T_m S)^\circ \bigr)
  \supset T_m S
  \,,
\end{equation*}
which means that $S$ is isotropic.
\end{proof}

Theorem~\ref{thm:CoisoPoiss} shows how reduction works for a hamiltonian Lie algebroid $A$ over a Poisson manifold: Since the anchor is tangent to $Z_\mathrm{cl}$, the Lie algebroid can be restricted to $Z_\mathrm{cl}$. If the leaf space of the characteristic distribution of $A|_{Z_\mathrm{cl}}$ is smooth, then it is a Poisson manifold \cite{MarsdenRatiu:1986}, called the \textbf{Poisson reduction of $A$}.

\section{Examples}
\label{sec:Examples}

\subsection{Bundles of Lie algebras}

A bundle of Lie algebras is a Lie algebroid with zero anchor. In this case, the opposite connection $\chD$, defined by Eq.~\eqref{eq:DcheckDef}, of every connection $D$ is zero. This shows that a bundle of Lie algebras  over a Poisson manifold is Poisson anchored with respect to every connection. This is the same as in the presymplectic case discussed in Section~6.1 of \cite{BlohmannWeinstein:HamLA}.

Since $\rho =  0$, a section $\mu$ of $A^*$ is a momentum section if and only if $\Pish\scal{D\mu}{a} = 0$ for every section $a$ of $A$. This is the case if and only if
\begin{equation*}
\begin{split}
  0 
  &= \bigl\langle \eta, \Pish \scal{D\mu}{a} \bigr\rangle
  = \bigl\langle \scal{D\mu}{a}, - \Pish\eta \bigr\rangle
  \\
  &= \scal{D_{- \Pish\eta}\mu}{a}
\end{split}
\end{equation*}
for all sections $a$ of $A$ and all sections $\eta$ of $A^*$. This shows that $\mu$ is a momentum section if and only if it is horizontal in the direction of the symplectic leaves of $(M, \Pi)$. One possible momentum section is $\mu = 0$.

For every momentum section we have $\Pi(\scal{D\mu}{a}, \scal{D\mu}{b}) = 0$, so that $\mu$ is bracket-compatible if and only if $\DA \mu = 0$. Since $\rho = 0$, $\DA \mu = -\langle \mu, [a,b] \rangle$. This shows that $\mu$ is bracket-compatible if and only if
\begin{equation*}
  \langle \mu, [a,b] \rangle = 0
\end{equation*}
for all sections $a$ and $b$ of $A$, which is the same condition as in the presymplectic case.

We conclude that a bundle of Lie algebras over a Poisson manifold is always hamiltonian with momentum section $\mu = 0$ and arbitrary connection $D$. If the fibre of $A$ over at least one point of every symplectic leaf of $M$ is semisimple, then $\mu = 0$ is the only momentum section.

\subsection{The cotangent Lie algebroid of a Poisson manifold}

Recall that the cotangent Lie algebroid of a Poisson manifold \cite[Thm.~4.1]{Vaisman:Lectures} is the vector bundle $A  = T^* M$ with anchor $\rho = - \Pish$ and Lie bracket
\begin{equation}
\label{eq:CotanBracket}
  [\alpha, \beta]
  =  
  - \Lie_{\Pish \alpha} \beta 
  + \Lie_{\Pish \beta} \alpha 
  + d\, \Pi(\alpha, \beta)
\end{equation}
for all 1-forms $\alpha$ and $\beta$ on $M$.

\subsubsection{Poisson anchored connections}

In the first step, we will determine the conditions for the cotangent Lie algebroid to be Poisson anchored with respect to some connection $D$ on $T^* M$. The $T^*M$-torsion~\eqref{eq:Atorsion} of $D$ is given by
\begin{equation}
\label{eq:TMstarTorsion}
  \TorTxM(\alpha, \beta) 
  = 
  - D_{\Pish \alpha} \beta
  + D_{\Pish \beta} \alpha
  - [\alpha, \beta]
  \,.
\end{equation}
The dual connection defined by~\eqref{eq:DPairing} is a (usual) connection on $TM$, so that there is the usual $TM$-torsion, which we denote by
\begin{equation}
\label{eq:TMTorsion}
  \TorTM(v,w) = D_v w - D_w v - [v,w]
  \,.
\end{equation}

\begin{Proposition}
\label{prop:DDPiTorsion}
Let $(M,\Pi)$ be a Poisson manifold and $D$ a connection on $T^*M$. Then
\begin{equation}
\label{eq:DDPiTorsion}
  (\chD_\gamma\Pi)(\alpha,\beta) 
  = -\bigl\langle \gamma, \TorTM(\Pish\alpha,\Pish\beta) \bigr\rangle
\end{equation}
for all $1$-forms $\alpha$, $\beta$, and $\gamma$.
\end{Proposition}
\begin{proof}
The left side of~\eqref{eq:DDPiTorsion} can be written as
\begin{equation*}
\begin{split}
  (\chD_\gamma\Pi)(\alpha, \beta)
  &=
  \chD_\gamma \bigl( \Pi(\alpha, \beta) \bigr)
  - \Pi(\chD_\gamma\alpha, \beta)
  - \Pi(\alpha, \chD_\gamma\beta)
  \\
  &=
  -\Pish\gamma \cdot \Pi(\alpha, \beta)
  + \langle \chD_\gamma\alpha, \Pish \beta \rangle
  - \langle \chD_\gamma\beta, \Pish \alpha \rangle
  \\
  &=
  -\Pish\gamma \cdot \Pi(\alpha, \beta)
  + \chD_\gamma\, \langle \alpha, \Pish \beta \rangle
  - \chD_\gamma\, \langle \beta, \Pish \alpha \rangle
  \\
  &{}\quad
  - \bigl\langle \alpha, \chD_\gamma(\Pish \beta) \bigr\rangle
  + \bigl\langle \beta, \chD_\gamma(\Pish \alpha) \bigr\rangle
  \\
  &=
  \Pish\gamma \cdot \Pi(\alpha, \beta)
  + \bigl\langle \alpha, 
    [\Pish\gamma, \Pish \beta] + \Pish(D_{\Pish\beta}\gamma)
    \bigr\rangle
  \\
  &{}\quad
  - \bigl\langle \beta, 
    [\Pish\gamma, \Pish \alpha] + \Pish(D_{\Pish\alpha}\gamma)
    \bigr\rangle
  \\
  &=
  \Pish\gamma \cdot \Pi(\alpha, \beta)
  + \bigl\langle \alpha, [\Pish\gamma, \Pish \beta] \bigr\rangle 
  - \bigl\langle \beta, [\Pish\gamma, \Pish \alpha] \bigr\rangle 
  \\
  &{}\quad
  - \bigl\langle D_{\Pish\beta}\gamma, \Pish\alpha \bigr\rangle
  + \bigl\langle D_{\Pish\alpha}\gamma, \Pish\beta \bigr\rangle
  \,.
\end{split}
\end{equation*}   
Pairing the Poisson condition $[\Pi,\Pi] = 0$ with $\alpha \wedge \beta \wedge \gamma$, we obtain the relation
\begin{equation*}
\begin{split}
  0 
  &= \Pish\alpha \cdot \Pi(\beta, \gamma)
  + \Pish\beta \cdot \Pi(\gamma, \alpha)
  + \Pish\gamma \cdot \Pi(\alpha, \beta)
  \\
  &{}\quad
  - \bigl\langle \alpha, [\Pish\beta, \Pish \gamma] \bigr\rangle
  - \bigl\langle \beta, [\Pish\gamma, \Pish \alpha] \bigr\rangle
  - \bigl\langle \gamma, [\Pish\alpha, \Pish \beta] \bigr\rangle
  \,.
\end{split}
\end{equation*}
Using this relation, we can write the right side of~\eqref{eq:DDPiTorsion} as
\begin{equation*}
\begin{split}
  (\chD_\gamma\Pi)(\alpha, \beta)
  &=
  - \Pish\alpha \cdot \Pi(\beta, \gamma)
  + \Pish\beta \cdot \Pi(\alpha, \gamma)
  + \bigl\langle \gamma, [\Pish\alpha, \Pish \beta] \bigr\rangle 
  \\
  &{}\quad
  - \langle D_{\Pish\beta}\gamma, \Pish\alpha \rangle
  + \langle D_{\Pish\alpha}\gamma, \Pish\beta \rangle
  \\
  &=
  - \Pish\alpha \cdot \langle\gamma, \Pish\beta\rangle
  + \langle D_{\Pish\alpha}\gamma, \Pish\beta \rangle
  \\
  &{}\quad
  + \Pish\beta \cdot \langle\gamma, \Pish\alpha \rangle
  - \langle D_{\Pish\beta}\gamma, \Pish\alpha \rangle
  + \bigl\langle \gamma, [\Pish\alpha, \Pish \beta] \bigr\rangle 
  \\
  &=
  - \bigl\langle\gamma, 
    D_{\Pish\alpha}\Pish\beta 
  - D_{\Pish\beta}\Pish\alpha 
  - [\Pish\alpha, \Pish \beta]
  \bigr\rangle
  \\
  &=
  - \bigl\langle\gamma, 
  \TorTM(\Pish\alpha, \Pish\beta)
  \bigr\rangle
  \,,
\end{split}
\end{equation*}   
where in the last step we have used the Definition~\eqref{eq:TMTorsion} of the $TM$-torsion.
\end{proof}

\begin{Corollary}
\label{cor:CotanPoissonAnchor}
The cotangent Lie algebroid $A = T^*M$ of a Poisson manifold $(M,\Pi)$ is Poisson anchored with respect to a connection on $T^*M$ if and only if the torsion of the dual connection on $TM$ vanishes on the characteristic distribution $\Pish(T^*M)$.
\end{Corollary}

\begin{Remark}
From an arbitrary connection $D$ on $TM$ we obtain a connection with vanishing torsion by
\begin{equation*}
  D'_v w = D_v w - \tfrac{1}{2} T_{D}(v,w)
  \,.
\end{equation*}
By Corollary~\ref{cor:CotanPoissonAnchor}, the dual connection to $D'$ is Poisson anchored.
\end{Remark}

Let us determine the space of connections whose torsion satisfies the condition of Corollary~\ref{cor:CotanPoissonAnchor}. The torsions of two connections $D'$ and $D$ on $TM$ are related by
\begin{equation}
\label{eq:TorsionChange}
  \TorTM'(v,w) = \TorTM(v,w) + \Gamma(v,w) - \Gamma(w,v)
  \,,
\end{equation}
where $\Gamma$ is the difference tensor of the connections, i.e.~the $C^\infty(M)$-bilinear map $\Gamma: \frakX(M) \times \frakX(M) \to \frakX(X)$ defined by
\begin{equation}
\label{eq:DDiffTensor}
 \Gamma(v,w) = D'_v w - D_v w \,.
\end{equation}
From Proposition~\ref{prop:DDPiTorsion} we deduce the following result.

\begin{Corollary}
Let the cotangent algebroid of a Poisson manifold $(M,\Pi)$ be Poisson anchored with respect to a connection $D$. Then it is Poisson anchored with respect to another connection $D'$ if and only if the difference tensor~\eqref{eq:DDiffTensor} satisfies
\begin{equation}
\label{eq:GammaPoiss}
  \bigl\langle \gamma, 
    \Gamma(\Pish \alpha, \Pish\beta) 
  - \Gamma(\Pish \beta, \Pish\alpha)   
  \bigr\rangle
  = 0
\end{equation}
for all 1-forms $\alpha$, $\beta$,and $\gamma$. The Poisson anchored connections form an affine space modelled on the vector space of morphisms of vector bundles $\Gamma: TM \times_M TM \to TM$ that satisfy~\eqref{eq:GammaPoiss}.
\end{Corollary}

\subsubsection{Momentum sections}

\begin{Theorem}
\label{thm:MomSec}
Let $(M,\Pi)$ be a Poisson manifold. Assume that there is a 1-form $\eta$ on $M$ such that the vector field $\mu = \Pish \eta$ is nowhere vanishing. Then, for a suitable Poisson-anchored connection, the cotangent Lie algebroid of $(M,\Pi)$ is weakly hamiltonian, with momentum section $\mu$.
\end{Theorem}
\begin{proof}
The proof is similar to that of Proposition~6.7 in \cite{BlohmannWeinstein:HamLA}. Let $D$ be a Poisson anchored connection. In general, $\mu$ will not be a $D$-momentum section. Let $D'$ be another connection. From~\eqref{eq:DPairing} it follows that the difference of the connections satisfy
\begin{equation*}
  \bigl\langle (D'_v - D_v)\alpha, w \bigr\rangle
  =  - \bigl\langle \alpha, (D'_v - D_v) w \bigr\rangle
  = - \bigl\langle \alpha, \Gamma(v,w) \bigr\rangle
  \,,
\end{equation*}
where $\Gamma$ is defined by~\eqref{eq:DDiffTensor}. The condition for $\mu$ to be a $D'$-momentum section is
\begin{equation}
\label{eq:MomSec1}
  - \Pish\alpha =  \Pish \scal{\alpha}{D'\mu}
  \,,
\end{equation}
for all 1-forms $\alpha$ on $M$. By pairing the right side with a 1-form $\beta$, we obtain
\begin{equation*}
\begin{split}
  \bigl\langle \beta, \Pish \langle\alpha, D'\mu \rangle 
  \bigr\rangle
  &=
  - \langle \alpha, D'_{\Pish\beta} \mu \rangle
  \\
  &=
  - \Pish\beta \cdot \langle \alpha, \mu \rangle
  + \langle D'_{\Pish\beta} \alpha,  \mu \rangle
  \\
  &=
  - \Pish\beta \cdot \langle \alpha, \mu \rangle
  + \langle D_{\Pish\beta} \alpha,  \mu \rangle
  + \bigl\langle (D'_{\Pish\beta} - D_{\Pish\beta})\alpha,  \mu \bigr\rangle
  \\
  &=
  - \Pish\beta \cdot \Pi(\eta, \alpha)
  + \Pi(\eta, D_{\Pish\beta} \alpha) 
  - \bigl\langle \alpha, 
  \Gamma(\Pish\beta, \mu) \bigr\rangle
  \\
  &=
  - (D_{\Pish\beta}\Pi)(\eta, \alpha)
  - \Pi(D_{\Pish\beta} \eta, \alpha) 
  - \bigl\langle \alpha, \Gamma(\Pish\beta, \Pish\eta)
    \bigr\rangle
  \,.
\end{split}    
\end{equation*}
With this relation~\eqref{eq:MomSec1} takes the form
\begin{equation}
\label{eq:MomSec2}
\begin{split}
  \bigl\langle \alpha, \Gamma(\Pish\beta, \Pish\eta) \bigr\rangle
  &= - (D_{\Pish\beta}\Pi)(\eta, \alpha)
  + \Pi( \alpha,  D_{\Pish\beta} \eta)
  - \langle\alpha, \Pish\beta \rangle
  \\
  &= B(\alpha, \Pish\beta)
  \,,    
\end{split}
\end{equation}
where 
\begin{equation*}
  B(\alpha, v)
  := 
  - (D_{v}\Pi)(\eta, \alpha)
  + \Pi( \alpha,  D_{v} \eta)
  - \langle\alpha, v \rangle
\end{equation*}
is $C^\infty(M)$-linear in $\alpha$ and $v$. 

Since $\mu$ is a nowhere-vanishing vector field on $M$, there exists a 1-form $\bar{\eta}$ on $M$, such that $\scal{\bar{\eta}}{\mu} = \Pi(\eta,\bar{\eta}) = 1$. Consider the expression 
\begin{equation*}
  C(\alpha,v,w) 
  := \scal{w}{\bar{\eta}}B(\alpha,v) 
  + \scal{v}{\bar{\eta}}B(\alpha,w)
  - \scal{v}{\bar{\eta}}\scal{w}{\bar{\eta}}B(\alpha,\Pish\eta)
  \,,
\end{equation*}
which is $C^\infty(M)$-linear in $\alpha \in \Omega^1(M)$ and $v,w \in \frakX^1(M)$, so that there is a unique $C^\infty(M)$-linear map $\Gamma: \frakX(M) \otimes \frakX(M) \to \frakX(M)$ that satisfies
\begin{equation*}
  \bigl\langle \alpha, \Gamma(v, w) \bigr\rangle
  = C(\alpha,v,w)
  \,.
\end{equation*}
for all $\alpha$, $v$, and $w$. Let now $D' = D + \Gamma$ for this choice of $\Gamma$.

Evaluating $C$ on $v = \Pish\beta$ and $w = \Pish\eta$ yields
\begin{equation*}
\begin{split}
  C(\alpha,\Pish\beta,\Pish\eta)
  &= 
    \langle \Pish\eta, \bar{\eta}\rangle B(\alpha, \Pish\beta)
  + \langle \Pish\beta, \bar{\eta}\rangle B(\alpha, \Pish\eta)
  \\
  &{}\quad
  - \langle \Pish\beta, \bar{\eta} \rangle 
    \langle \Pish\eta, \bar{\eta} \rangle
    B(\alpha, \Pish\eta)
  \\
  &= B(\alpha, \Pish\beta)
  \,.
\end{split}    
\end{equation*}
This shows that $\Gamma$ satisfies~\eqref{eq:MomSec2}, which means that $D'$ satisfies~\eqref{eq:MomSec1}, so that $\mu$ is a $D'$-momentum map. Furthermore, $C$ is symmetric in $v$ and $w$, so that $\Gamma(v,w)$ is symmetric in $v$ and $w$. It follows from~\eqref{eq:TorsionChange}, that $D'$ has the same $TM$-torsion as $D$. (In particular, if $D$ is torsion-free, so is $D'$.) Since by assumption $D$ was assumed to be Poisson anchored, we conclude by Proposition~\ref{prop:DDPiTorsion} that $D'$ is Poisson anchored.
\end{proof}

\begin{Corollary}
\label{cor:RegPoiss}
Let $(M,\Pi)$ be a regular Poisson manifold with non-zero Poisson structure. If the characteristic distribution $\Pish(T^* M) \subset TM$ admits a nowhere vanishing section, then the cotangent Lie algebroid of $(M,\Pi)$ is weakly hamiltonian.
\end{Corollary}
\begin{proof}
Let the nowhere vanishing section of $\Pish(T^* M)$ be denoted by $\mu$. Since $\Pi$ is regular, every point $m \in M$ has a neighborhood $U$ with Darboux coordinates $(q^1, \ldots, q^r, p_1, \ldots, p_r, y^1, \ldots, y^{\dim M - 2r})$  in which the Poisson bivector field takes the form $\Pi = \frac{\partial}{\partial q^i} \wedge \frac{\partial}{\partial p_i}$. In local coordinates the section $\mu$ takes the form $\mu = \alpha^i \frac{\partial}{\partial q^i} + \beta_i \frac{\partial}{\partial p_i}$. It is easy to check that $\eta = \beta_i dq^i - \alpha^i dp_i$ satisfies $\mu = \Pish\eta$ on $U$. Since the map $\Pish: \Omega^1(M) \mapsto \frakX(M)$ is $C^\infty(M)$-linear, we can use a standard partition of unity argument to show that there is a globally defined form $\eta \in \Omega^1(M)$ such that $\mu = \Pish\eta$. Now we can apply Theorem~\ref{thm:MomSec}.
\end{proof}

The assumption of Theorem~\ref{thm:MomSec} that $\mu = \Pish\eta$ be nowhere vanishing implies that $\Pi$ is nowhere vanishing. This is not necessary for the cotangent Lie algebroid to have a momentum map, as the Example~\ref{ex:LiePoisson} of Lie-Poisson manifolds shows. Even when $\Pi$ is regular, as is assumed in Corollary~\ref{cor:RegPoiss}, a momentum section need not be nowhere-vanishing, nor of the form $\Pish\eta$, as the following example shows.

\begin{Example}
\label{ex:Zeromu}
Let $M = \bbR^3$ with Poisson structure $\Pi=\del_x\wedge\del_y$. The trivial connection $D$ on $T^*M$ is Poisson anchored because its dual, the trivial connection on $TM$, is torsion-free. Consider $\mu := -x \del_x -y \del_y$. We have
\begin{equation*}
  \scal{D\mu}{\alpha}
  = -\alpha_x dx -\alpha_y dy = -\alpha +\alpha_z dz.
\end{equation*}
for all 1-forms $\alpha = \alpha_x dx +\alpha_y dy + \alpha_z dz$. It follows that
\begin{equation*}
  \Pish \scal{D\mu}{\alpha}
  = \Pish( -\alpha +\alpha_z dz) 
  = - \Pish(\alpha)
  \,.
\end{equation*}
This shows that $\mu$ is a momentum section, even though it vanishes on $\{(0,0)\} \times \bbR$.

Consider now $\mu' := \mu - \del_z$. Since $\mu'$ and $\mu$ differ by a constant vector field, their covariant derivatives are equal, $D\mu = D\mu'$, which implies that $\mu'$ is a momentum section, too. Since $\mu'$ is not tangent to $\Pish(T^* M)$, it cannot be of the form $\Pish\eta$.
\end{Example}

\begin{Remark}
The vector field $\mu$ in Example~\ref{ex:Zeromu} is the negative Euler vector field in Darboux coordinates of the symplectic leaves. For the symplectic case it was shown in~\cite[Thm.~6.9]{BlohmannWeinstein:HamLA} that the vector field of a momentum section may have isolated zeros at which it has the form of the negative Euler vector field in Darboux coordinates. Such a vector field exists on a compact manifold if and only if the Euler characteristic is nonnegative. (See Theorem~\ref{thm:WeaklyHamEquiv} below.)
\end{Remark}

\subsubsection{Bracket-compatibility of momentum sections}

The Lie algebroid cohomology of the cotangent Lie algebroid of a Poisson manifold is the same as the Poisson cohomology. In particular, the differential of a vector field $v$ on $M$, viewed as Lie algebroid 1-form on $T^* M$, is given by
\begin{equation*}
\begin{split}
  d_{T^* M} v 
  &= d_\Pi v = [v,\Pi] \\
  &= \Lie_v \Pi
  \,.
\end{split}
\end{equation*}

\begin{Proposition}
\label{prop:Liouville}
Let $\mu \in \frakX(M)$ be a momentum section of the cotangent Lie algebroid of a Poisson manifold $(M,\Pi)$. Then $\mu$ is bracket-compatible if and only if it is a Liouville vector field, that is $\Lie_\mu \Pi = \Pi$.
\end{Proposition}
\begin{proof}
The right side of the condition (H3$_\mathrm{Poi}$) of bracket-compatibility of $\mu$ is given by 
\begin{equation*}
\begin{split}
  \Pi \bigl(
     \scal{D\mu}{\alpha}, \scal{D\mu}{\beta} \bigr)
  &= \bigl\langle \Pish\scal{D\mu}{\alpha},
    \scal{D\mu}{\beta} \bigr\rangle
  \\
  &= \bigl\langle -\Pish\alpha, \scal{D\mu}{\beta}
     \bigr\rangle
  \\
  &= \bigl\langle \alpha, \Pish\scal{D\mu}{\beta}
     \bigr\rangle
  \\
  &= \scal{\alpha}{-\Pish\beta}
  \\
  &= \Pi(\alpha,\beta)
\end{split}
\end{equation*}
for all $\alpha,\beta \in \Omega^1(M)$, where we have used that $\Pish\langle D\mu, \alpha \rangle = - \Pish\alpha$ because $\mu$ is a momentum section. It follows that the condition for bracket compatibility is
\begin{equation*}
  \Lie_\mu \Pi = d_\Pi \mu = d_{T^* M} \mu = \Pi
  \,,
\end{equation*}
which was the claim.
\end{proof}

\begin{Proposition}
\label{prop:HamLA}
Let $(M,\Pi)$ be a Poisson manifold. Assume that there is a 1-form $\eta$ on $M$ such that the vector field $\mu = \Pish \eta$ is nowhere vanishing and Liouville, $\Lie_\mu \Pi = \Pi$. Then, for a suitable Poisson-anchored connection, the cotangent Lie algebroid of $(M,\Pi)$ is hamiltonian, with momentum section $\mu$.
\end{Proposition}
\begin{proof}
This follows from Theorem~\ref{thm:MomSec} and Proposition~\ref{prop:Liouville}.
\end{proof}

The condition of bracket-compatibility of a momentum section of the form $\mu = \Pish\eta$ can be expressed directly in terms of $\eta$ as follows.

\begin{Proposition}
Let $\eta$ be a 1-form on the Poisson manifold $(M,\Pi)$. Then the vector field $\mu = \Pish\eta$ is Liouville if and only if
\begin{equation}
\label{eq:etaLiouville}
  (d\eta)\bigl(\Pish\alpha, \Pish\beta \bigr)
  = -\Pi(\alpha, \beta)
\end{equation}
for all $\alpha, \beta \in \Omega^1(M)$.
\end{Proposition}
\begin{proof}
The anchor $\rho := -\Pish: T^* M \to TM$ is a morphism of Lie algebroids. On Lie algebroid forms it induces a map
\begin{equation*}
  \rho^* : \Omega(M) \to \frakX(M)
  \,,
\end{equation*}
that intertwines the differentials, $d_{T^* M} \rho^* = \rho^* d$, where $d$ is the de Rham differential. On a 1-form $\eta$ this is given explicitly by
\begin{equation*}
\begin{split}
  (d_{T^* M} \rho^*\eta)(\alpha, \beta)
  &=
  (\rho^* d\eta)(\alpha, \beta) 
  = (d\eta)(-\Pish \alpha, - \Pish \beta)
  \\
  &= (d\eta)(\Pish \alpha, \Pish \beta)
\end{split}
\end{equation*}
for all $\alpha, \beta \in \Omega^1(M)$. Since $\mu = -\rho^*\eta$ and since $d_{T^* M} = d_\Pi$, it follows that the condition for $\mu$ to be Liouville, $d_\Pi \mu = \Pi$, is given by Equation~\eqref{eq:etaLiouville}.
\end{proof}

\begin{Proposition}
The cotangent Lie algebroid of every regular Poisson manifold is locally hamiltonian.
\end{Proposition}

\begin{proof}
It follows from the splitting theorem \cite{Weinstein:1983} that every point $m \in M$ has a neighborhood with coordinates
\begin{equation*}
  (q^1,\dots,q^r, p_1,\dots,p_r,y_{2r+1},\dots,y_n)
\end{equation*}
mapping $m$ to $0$, in which $\Pi$ has the form
\begin{equation*}
  \Pi = \pder{q^i}\wedge\pder{p_i}
  \,.
\end{equation*}
Let $\eta := q^i dp_i + dp_1$. We have 
\begin{equation*}
  \mu = \Pish\eta = 
  -q^i \frac{\partial}{\partial q^i} - \frac{\partial}{\partial q^1}
  \,,
\end{equation*}
which is nowhere vanishing in a possibly smaller neighborhood of $m = 0$.  It is straightforward to check that the differential $d\eta = dq^i \wedge dp^i$ satisfies~\eqref{eq:etaLiouville}, so that the assumptions of Proposition~\ref{prop:HamLA} are satisfied for the cotangent Lie algebroid restricted to the coordinate neighborhood.
\end{proof}

\begin{Remark}
Using Proposition~\ref{prop:H3PoissonTorsion}, we see that a momentum section $\mu$ of the cotangent Lie algebroid of $(M,\Pi)$ is bracket-compatible if and only if
\begin{equation*}
  \bigl\langle \TorTxM(\alpha,\beta), \mu \bigr\rangle 
  = -\Pi(\alpha,\beta)
\end{equation*}
for all 1-forms $\alpha$ and $\beta$. This implies that, if $\Pi$ is non-vanishing at a point, neither $\mu$ nor $\TorTxM$ can vanish at that point.
\end{Remark}

\subsection{The tangent Lie algebroid over a Poisson manifold}

We now consider the usual tangent Lie algebroid $A = TM$, where $M$ is a Poisson manifold. In this case, a connection is related to its opposite $A$-connection, which is now a connection in the usual sense, by
\begin{equation}
\label{eq:NablaCheck}
  \check{D}_v w = [v,w] + D_w v = D_v w - \TorTM(v,w)
  \,,
\end{equation}
for all vector fields $v$ and $w$ on $M$. If $D$ is torsion-free, it follows that $TM$ is Poisson anchored with respect to $D$ if and only if $D \Pi= 0$. Such connections are called Poisson connections and will be considered in more detail in Section~\ref{sec:PoissonConnections}. In the general case, we have the following surprising relation to the torsion of the cotangent Lie algebroid of $(M,\Pi)$.

\begin{Theorem}
\label{thm:TangentTorsion}
The tangent Lie algebroid of a Poisson manifold $(M,\Pi)$ is Poisson anchored with respect to a connection on $TM$ if and only if the $T^*M$-torsion~\eqref{eq:TMstarTorsion} of the dual connection on $T^*M$ vanishes.
\end{Theorem}
\begin{proof}
For the covariant derivative of the connection we have
\begin{equation*}
\begin{split}
  \langle D_{\Pish\alpha} \beta, v \rangle
  &=
  \Pish\alpha \cdot \langle \beta, v \rangle
  - \langle \beta, D_{\Pish\alpha} v \rangle
  \\
  &=
    \langle \Lie_{\Pish\alpha} \beta, v \rangle
  + \langle \beta, \Lie_{\Pish\alpha} v \rangle
  - \langle \beta, D_{\Pish\alpha} v \rangle
  \\
  &=
    \langle \Lie_{\Pish\alpha} \beta, v \rangle
  - \bigl\langle \beta, D_v (\Pish\alpha) \bigr\rangle
  - \bigl\langle \beta, \TorTM(\Pish\alpha, v) \bigr\rangle
  \\
  &=
    \langle \Lie_{\Pish\alpha} \beta, v \rangle
  - v \cdot \langle \beta, \Pish\alpha \rangle 
  + \langle D_v \beta, \Pish\alpha \rangle
  - \bigl\langle \beta, \TorTM(\Pish\alpha, v) \bigr\rangle
  \\
  &=
  \bigl\langle \Lie_{\Pish\alpha} \beta
  - d\,\Pi(\alpha,\beta), v \bigr\rangle
  + \Pi(\alpha, D_v \beta)
  - \bigl\langle \beta, \TorTM(\Pish\alpha, v) \bigr\rangle
  \,.
\end{split}
\end{equation*}
With this relation we obtain for the $T^*M$-torsion the following equation:
\begin{equation*}
\begin{split}
  \bigl\langle \TorTxM(\alpha,\beta), v \bigr\rangle
  &= 
  -\langle D_{\Pish \alpha} \beta, v \rangle
  + \langle D_{\Pish \beta} \alpha, v \rangle
  - \langle [\alpha, \beta], v \rangle
  \\
  &= 
  \bigl\langle 
  - \Lie_{\Pish \alpha} \beta
  + \Lie_{\Pish \beta} \alpha
  + d\,\Pi(\alpha,\beta)
  - [\alpha, \beta]
  , v \bigr\rangle
  \\
  &{}\quad
  + \bigl\langle d\,\Pi(\alpha,\beta), v \bigr\rangle
  - \Pi(\alpha,D_v\beta)
  + \Pi(\beta,D_v\alpha)
  \\
  &{}\quad
  + \bigl\langle \beta, \TorTM(\Pish\alpha, v) \bigr\rangle
  - \bigl\langle \alpha, \TorTM(\Pish\beta, v) \bigr\rangle
  \,.
\end{split}
\end{equation*}
The first line on the right side vanishes by Equation~\eqref{eq:CotanBracket}. The second line is equal to $(D_v \Pi)(\alpha,\beta)$. We conclude that the $T^*M$-torsion~\eqref{eq:TMstarTorsion} of the dual connection and the $TM$-torsion~\eqref{eq:TMTorsion} of the connection on $TM$ are related by
\begin{equation}
\label{eq:DualTorsionATorsionPoisson}
  \bigl\langle \TorTxM(\alpha,\beta), v \bigr\rangle
  = 
  (D_v\Pi)(\alpha,\beta)
  + \bigl\langle \beta, \TorTM(\Pish\alpha,v) \bigr\rangle
  - \bigl\langle \alpha, \TorTM(\Pish\beta,v) \bigr\rangle
\end{equation}
for all 1-forms $\alpha,\beta$ and all vector fields $v$ on $M$. 

By the derivation property of the covariant derivatives, we have
\begin{equation*}
\begin{split}
  \chD_v (u \wedge w)
  &=
  \chD_v u \wedge w + u \wedge \chD_v w
  \\
  &=
  \bigl(D_v w + \TorTM(v,u)\bigr) \wedge w 
  + u \wedge \bigl( D_v w + \TorTM(v,w) \bigr)
  \\
  &=
  D_v (u \wedge w)
  + \TorTM(v,u) \wedge w + u \wedge \TorTM(v,w)
  \,.
\end{split}
\end{equation*}
It follows that the action of the covariant derivative and its opposite covariant derivative on the bivector $\Pi$ are related by
\begin{equation*}
  (\chD_v\Pi)(\alpha,\beta) 
  = 
  (D_v\Pi)(\alpha,\beta)
  - \bigl\langle \alpha, \TorTM(\Pish\beta,v) \bigr\rangle
  + \bigl\langle \beta, \TorTM(\Pish\alpha,v) \bigr\rangle
  \,.
\end{equation*}
Inserting Equation~\eqref{eq:DualTorsionATorsionPoisson}, we obtain
\begin{equation*}
  (\chD_v\Pi)(\alpha,\beta) 
  = \bigl\langle \TorTxM(\alpha,\beta), v \bigr\rangle
  \,,
\end{equation*}
for all 1-forms $\alpha$ and $\beta$ and all vector fields $v$ on $M$. We conclude that $\chD\Pi = 0$ if and only if $\TorTxM = 0$.  
\end{proof}

The condition for $\mu \in \Omega^1(M)$ to be a $D$-momentum section of the tangent Lie algebroid is 
\begin{equation*}
  \iota_{\langle D\mu, v \rangle} \Pi = v
\end{equation*}
for all vector fields $v$. This condition can be written as
\begin{equation*}
  \Pish \circ \langle D\mu, \Empty \rangle = id_{TM}
  \,,
\end{equation*}
where $\langle D\mu, \Empty \rangle: TM \to T^* M$, $v \mapsto \langle D\mu, v\rangle$. We conclude that $\Pi$ must be symplectic for $TM$ to have a momentum section.

\subsection{The symplectic case revisited}

Assume that the Poisson bivector $\Pi$ on $M$ is non-degenerate. Let $\omega = \Pi^{-1}$ be the symplectic form. In Proposition~\ref{prop:PoissPresympEquiv}, we have seen that in this case the axioms (H1)-(H3) of the Poisson and the presymplectic case are equivalent. Moreover, the anchor of the cotangent Lie algebroid $\rho = -\Pish: T^* M \to TM$ is an isomorphism of Lie algebroids. Therefore, we can equivalently consider the following three Lie algebroids:
\begin{itemize}

\item the cotangent Lie algebroid of the Poisson manifold $(M,\Pi)$,

\item the tangent Lie algebroid over the Poisson manifold $(M,\Pi)$,

\item the tangent Lie algebroid over the symplectic manifold $(M,\omega)$.

\end{itemize}

Each one of the axioms (H1)-(H3) of hamiltonian Lie algebroids holds if it is, equivalently, satisfied for all three Lie algebroids. It follows that the topological conditions for (weakly) hamiltonian Lie algebroids over symplectic manifolds of Theorem~6.9 in \cite{BlohmannWeinstein:HamLA} apply to the Poisson case:

\begin{Theorem}
\label{thm:WeaklyHamEquiv}
Let $(M,\Pi)$ be a non-degenerate Poisson (i.e.~symplectic) manifold. The following are equivalent:
\begin{itemize}

\item[(i)] The cotangent Lie algebroid of $(M,\Pi)$ is weakly hamiltonian.

\item[(ii)] The tangent Lie algebroid over $(M,\Pi)$ is weakly hamiltonian.

\item[(iii)] The tangent Lie algebroid over $(M,\omega)$ is weakly hamiltonian.

\item[(iv)] $M$ is either non-compact or compact with non-negative Euler characteristic.

\end{itemize}
\end{Theorem}

\begin{Remark}
The assumption of Theorem~\ref{thm:MomSec} that there is a nowhere vanishing vector field on $M$ is equivalent to $M$ being non-compact or compact with zero Euler characteristic. This shows that, in the non-degenerate case, Theorem~\ref{thm:WeaklyHamEquiv} implies 
Theorem~\ref{thm:MomSec}.
\end{Remark}

\begin{Theorem}
\label{thm:HamEquiv}
Let $(M,\Pi)$ be a non-degenerate Poisson manifold. The following are equivalent:
\begin{itemize}

\item[(i)] The cotangent Lie algebroid of $(M,\Pi)$ is hamiltonian.

\item[(ii)] The tangent Lie algebroid over $(M,\Pi)$ is hamiltonian.

\item[(iii)] The tangent Lie algebroid over $(M,\omega)$ is hamiltonian.

\item[(iv)] The bivector field $\Pi$ is exact in Poisson cohomology.

\item[(v)] The symplectic form $\omega$ is exact.

\end{itemize}
\end{Theorem}
\begin{proof}
The theorem follows from Proposition~6.15 and Theorem~6.18 in \cite{BlohmannWeinstein:HamLA}. Theorem~6.18 relies on the recent work by Stratmann \cite{Stratmann:2020b}, and by Karshon and Tang \cite{KarshonTang:2021}, who proved that if $\omega =d\lambda$ is an exact symplectic form on a (necessarily non-compact) manifold $M$, then we can remove rays containing the zeros of $\lambda$ by a symplectomorphism and obtain a primitive of $\omega$ that is nowhere vanishing.
\end{proof}

There is an interesting subtlety about the equivalence of the cotangent Lie algebroid of and the tangent Lie algebroid over a Poisson manifold. On the one hand, it follows from  Corollary~\ref{cor:CotanPoissonAnchor} that $T^* M$ is Poisson anchored by a connection on $T^*M$ if the $TM$-torsion of the dual connection on $TM$ vanishes. On the other hand, Theorem~\ref{thm:TangentTorsion} states that $TM$ is Poisson anchored by a connection on $TM$ if the $T^*M$-torsion of the dual connection on $T^*M$ vanishes. However, according to~\eqref{eq:DualTorsionATorsionPoisson}, the vanishing of the $T^* M$-torsion is equivalent to the vanishing of the $TM$-torsion of the dual connection only if the connection is Poisson. This seeming contradiction is resolved by observing that the isomorphism of Lie algebroids $-\Pish: T^*M \to TM$ does \emph{not} map a connection $D$ on $T^* M$ to its dual connection on $TM$, but rather to the isomorphic connection $D'_v w = (\Pish \circ D_v \circ \omega_\flat)w$. This observation is made more precise by the following statement.

\begin{Proposition}
\label{prop:SympHLAs1}
Let $\omega$ be a symplectic form on $M$ and $\Pi = \omega^{-1}$ its Poisson bivector field. Let $D$ be a connection on $TM$ and $D'$ another connection on $TM$ defined by
\begin{equation}
\label{eq:omegaPiConnect}
  D'_v u = D_v u + \Pish(\iota_u D_v \omega)
  \,,
\end{equation}
for all vector fields $u$ and $v$. The following are equivalent:
\begin{itemize}

\item[(i)] $D$ has vanishing torsion.

\item[(ii)] The cotangent Lie algebroid of $(M,\Pi)$ is Poisson anchored by the dual connection of $D$.

\item[(iii)] The tangent Lie algebroid over $(M,\Pi)$ is Poisson anchored by $D'$.

\item[(iv)] The tangent Lie algebroid over $(M,\omega)$ is symplectically anchored by $D'$.

\end{itemize}
\end{Proposition}
\begin{proof}
Let $D$ be a connection on $TM$. In the non-degenerate case Corollary~\ref{cor:CotanPoissonAnchor} states that $T^* M$ is Poisson anchored with respect to the dual connection if and only if the $TM$-torsion of $D$ vanishes. This shows that (i) and (ii) are equivalent.

The isomorphism of Lie algebroids $-\Pish: T^* M \to TM$, maps the connection on $T^* M$ to a connection $D'$ on $TM$ given by  
\begin{equation*}
  D'_v u := (\Pish \circ D_v\circ \omega_\flat ) u
  \,.
\end{equation*}
It follows that $TM$ is Poisson anchored by $D'$ if and only if $T^* M$ is Poisson anchored by $D'$. The covariant derivative can be expressed explicitly as
\begin{equation*}
\begin{split}
  \langle \alpha, D'_v u \rangle
  &= \bigl\langle \alpha, 
  \Pish\bigl(D_v (\omega_\flat u) \bigr)\bigr\rangle
  \\
  &= 
  - \bigl\langle D_v (\omega_\flat u), \Pish \alpha \bigr\rangle
  \\
  &= 
  - v \cdot \langle \omega_\flat u, \Pish \alpha \rangle
  + \bigl\langle \omega_\flat u, D_v (\Pish \alpha) \bigr\rangle
  \\
  &= 
  - v \cdot \omega(u , \Pish \alpha)
  + \omega\bigl( u, D_v (\Pish \alpha) \bigr)
  \\
  &= 
  - \omega(D_v u, \Pish\alpha) 
  - (D_v \omega)( u, \Pish \alpha)
  \\
  &= 
  \langle\alpha, D_v u \rangle 
  - (D_v \omega)( u, \Pish \alpha)
  \,,
\end{split}
\end{equation*}
for all 1-forms $\alpha$. The second term can be written as
\begin{equation*}
\begin{split}
  - (D_v \omega)(u, \Pish\alpha)
  &=
  - \langle \iota_u D_v \omega, \Pish\alpha \rangle
  =
  \bigl\langle \alpha, \Pish(\iota_u D_v \omega)
  \bigr\rangle
  \,.
\end{split}
\end{equation*}
From the last two equations we see that $D'_v u$ is given by Eq.~\eqref{eq:omegaPiConnect}. We conclude that $TM$ is Poisson anchored by $D'$ if and only if $T^* M$ is Poisson anchored by the dual connection of $D$, that is (ii) and (iii) are equivalent. It follows from Proposition~\ref{prop:PoissPresympEquiv} that (iii) and (iv) are equivalent.
\end{proof}

\begin{Proposition}
\label{prop:SympHLAs2}
Let $\omega$ be a symplectic form on $M$ and $\Pi = \omega^{-1}$ its Poisson bivector field. Let $D$ be a connection on $TM$ and $D'$ the connection on $TM$ defined by~\eqref{eq:omegaPiConnect}. Let $n$ be a vector field on $M$ and let 
\begin{equation*}
  \mu := -\Pish n \,.
\end{equation*}
The following are equivalent:
\begin{itemize}

\item[(i)] $D_v n = -v$ for all $v \in \frakX(M)$.

\item[(ii)] $n$ is a momentum section of the cotangent Lie algebroid of $(M,\Pi)$ with respect to the dual connection of $D$.

\item[(iii)] $\mu$ is a $D'$-momentum section of the tangent Lie algebroid over $(M,\Pi)$.

\item[(iv)] $\mu$ is a $D'$-momentum section of the tangent Lie algebroid over $(M,\omega)$. 

\end{itemize}
\end{Proposition}
\begin{proof}
The proof is analogous to the proof of Proposition~\ref{prop:SympHLAs1}. Since $\Pish$ is by assumption an isomorphism, it follows from~\eqref{eq:MomSec1} that a vector field $n$ is a momentum section of $T^* M$ with respect to a connection $D$ if and only if
\begin{equation*}
  D_v n  = - v
  \,.
\end{equation*}
This shows that (i) and (ii) are equivalent. The isomorphism of Lie algebroids $-\Pish: T^*M \to TM$ maps $D$ to $D'$ and the vector field $n$ to the 1-form $\mu = -\Pish n$. It follows that (ii) and (iii) are equivalent. The equivalence of (iii) and (iv) follows from Proposition~\ref{prop:PoissPresympEquiv}.    
\end{proof}

It is remarkable that condition (i) of Proposition~\ref{prop:SympHLAs1} and condition (i) of Proposition~\ref{prop:SympHLAs2} both do not depend on the Poisson or symplectic structure. This is consistent with the topological condition (iv) of Theorem~\ref{thm:WeaklyHamEquiv} which was proved in \cite[Thm.~6.9]{BlohmannWeinstein:HamLA}. We finally turn to the bracket-compatibility.

\begin{Proposition}
\label{prop:SympHLAs3}
Let $\omega$ be a symplectic form on $M$ and $\Pi = \omega^{-1}$ its Poisson bivector field. Let $D$ be a connection on $TM$ and $D'$ the connection on $TM$ defined by~\eqref{eq:omegaPiConnect}. Let $n$ be a vector field and $\mu = -\Pish n$. Assume that $D$ has vanishing torsion and that $D_v n = -v$ for all $v \in \frakX(M)$ so that the equivalent statements of Propositions~\ref{prop:SympHLAs1} and \ref{prop:SympHLAs2} hold. The following are equivalent:
\begin{itemize}

\item[(i)] $D_n \Pi = - \Pi$.

\item[(ii)] The momentum section $n$ of the cotangent Lie algebroid of $(M,\Pi)$ with respect to the dual connection of $D$ is bracket-compatible.

\item[(iii)] The $D'$-momentum section $\mu$ of the tangent Lie algebroid over $(M,\Pi)$ is bracket compatible.

\item[(iv)] The $D'$-momentum section $\mu$ of the tangent Lie algebroid over $(M,\omega)$ is bracket compatible.

\end{itemize}
\end{Proposition}
\begin{proof}
In the non-degenerate case, Equation~\eqref{eq:MomSec1} implies that $\langle \alpha, Dn \rangle = -\alpha$. It follows from Proposition~\ref{prop:H3PoissonTorsion} that $n$ is bracket-compatible if and only if
\begin{equation}
\label{eq:TorsionPi}
\begin{split}
  \langle T_{T^*M}(\alpha,\beta), n \rangle
  &= 
  -\Pi\bigl( 
    \langle\alpha, Dn \rangle,
    \langle\beta, Dn \rangle
  \bigr)
  \\
  &= 
  -\Pi(\alpha, \beta)
  \,,
\end{split}    
\end{equation}
for all 1-forms $\alpha$ and $\beta$. The $T^* M$-torsion can be expressed by Equation~\eqref{eq:DualTorsionATorsionPoisson} as
\begin{equation*}
  \bigl\langle \TorTxM(\alpha,\beta), v \bigr\rangle
  = (D_v\Pi)(\alpha,\beta)
  \,,
\end{equation*}
where we have used that, by assumption, $\TorTM = 0$. It follows that $n$ is bracket-compatible if and only if $D_n \Pi = - \Pi$, that is (i) and (ii) are equivalent. That $-\Pish: T^* M \to TM$ is an isomorphism of Lie algebroids implies that (ii) and (iii) are equivalent. The equivalence of (iii) and (iv) follows from Proposition~\ref{prop:PoissPresympEquiv}.
\end{proof}

\subsection{Relation with Poisson connections}
\label{sec:PoissonConnections}

In \cite[Sec.~1(e), p.~65]{BayenEtAl:1978} a \textbf{Poisson connection} was defined to be a torsion-free connection $D$ on the tangent bundle of a Poisson manifold $(M,\Pi)$ such that $D\Pi = 0$. In this terminology we can rephrase the condition for a tangent Lie algebroid to be Poisson anchored by a torsion-free connection in the following way.

\begin{Proposition}
\label{prop:PoissConn1}
Let $D$ be a torsion-free connection on the tangent bundle of a Poisson manifold $(M,\Pi)$. The following are equivalent:
\begin{itemize}

\item[(i)] $TM$ is Poisson anchored by $D$.

\item[(ii)] $D$ is Poisson.

\end{itemize}
\end{Proposition}
\begin{proof}
Since $D$ is torsion free, Equation~\eqref{eq:NablaCheck} shows that $\chD = D$. It follows that condition (i), which is $\chD \Pi = 0$, is equivalent to $D\Pi = 0$.
\end{proof}

It was shown in Theorem~2.20 of \cite{Vaisman:Lectures} that a Poisson connection exists if and only if $\Pi$ is regular. As corollary of Proposition~\ref{prop:PoissConn1} it then follows that the tangent algebroid is Poisson anchored by a torsion-free connection if and only if $\Pi$ is regular. 

It follows from Corollary~\ref{cor:CotanPoissonAnchor} that the cotangent Lie algebroid of a Poisson manifold is Poisson anchored by the dual of every torsion-free connection, in particular by the dual of a Poisson connection. However, Proposition~\ref{prop:PoissConn1} and Theorem~\ref{thm:WeaklyHamEquiv} imply that the $T^* M$-torsion of the dual of a Poisson connection vanishes. It then follows from Equation~\eqref{eq:TorsionPi}, that $T^*M$ cannot have a bracket-compatible momentum section with respect to the dual of a Poisson connection unless $\Pi = 0$.

In \cite[Sec.~III]{Vaisman:1991} the notion of contravariant derivatives on vector bundles over a Poisson manifold was introduced. They are the same as $T^* M$-connections of the cotangent Lie algebroid of the Poisson manifold. In \cite[Sec.~2.5]{Fernandes:2000} a $T^* M$-connection $\nabla$ on $TM$ was called Poisson if $\nabla \Pi = 0$. In this terminology we can rephrase the condition for $T^*M$ to be Poisson anchored in the following way.

\begin{Proposition}
\label{prop:PoissConn2}
Let $(M,\Pi)$ be a Poisson manifold. Let $D$ be a connection on $T^* M$. The following are equivalent:
\begin{itemize}

\item[(i)] The cotangent Lie algebroid is Poisson anchored by $D$.

\item[(ii)] The opposite $T^* M$-connection $\check{D}$ is Poisson.

\end{itemize}
\end{Proposition}

\section{Interpretation in terms of Poisson structures on $A^*$}
\label{sec:Cabrera}

After a talk on hamiltonian Lie algebroids over presymplectic manifolds given in Banff in 2017 by one of the authors, Alejandro Cabrera proposed interpreting the compatibility conditions between a symplectic structure on $M$, a connection on a Lie algebroid $A$ over $M$, and a section of $A^*$ in terms of a certain pair of Poisson structures on $A^*$.  Since the natural setting for his idea involves a (not necessarily symplectic) Poisson structure on $M$, the present paper is a natural setting for investigating his proposal.  This section is devoted to just that.

Given a connection on $A$, the dual connection on $A^*$ allows one to lift the Poisson structure $\Pi$ on $M$ to a horizontal bivector field $\widehat\Pi$ on the manifold $A^*$.  There, one also has the natural Lie-Poisson structure $\Pi_A$ such as one has on the dual of any Lie algebroid.  In this context, a section $\mu$ of $A^*$ can be seen as a map between manifolds equipped with bivector fields.

Now we can impose compatibility conditions such as $[\widehat\Pi,\Pi_A]=0$ or that $\widehat\Pi + \Pi_A$ be a Poisson structure (somewhat different, because $\widehat\Pi$ might not be a Poisson structure if the connection is not flat).   Then we can compare these compatibility conditions with ours.  Given such compatibility, to see whether a section $\mu$ of $A^*$ is a (bracket compatible) momentum section, we can consider the behavior of $\mu$ with respect to the bivector fields on its domain and codomain.

\begin{Example}
For an action Lie algebroid with the trivial connection, the sections with covariant derivative zero correspond to elements of $\frakg$.  The horizontal lift  $\widehat\Pi$ is just the product of $\Pi$ with the zero Poisson structure on $\frakg^*$.  On the other hand, the algebroid Lie-Poisson structure $\Pi_A$ (which ignores the Poisson structure on $M$) 
is the sum of the algebra Lie-Poisson structure $\Pi_A '$ in the $\frakg^*$ direction, which clearly commutes with $\widehat\Pi$, and 
the ``mixed structure'' $\Pi_A''$ defined by the action (i.e.~the anchor of the action Lie algebroid). Either version of the condition of Cabrera thus reduces to  $[\widehat\Pi,\Pi_A'']=0$, which (as in the definition itself of Poisson anchoring), depends only on the action and not on the bracket structure in the Lie algebra.    

Now it is not hard to see that  $[\widehat\Pi,\Pi_A'']=0$ if and only if the Lie derivative of $\Pi$ by the anchor applied to any constant section of the action algebroid is 0, which is just the condition that the algebroid be Poisson anchored with respect to the trivial connection, i.e., that the Lie algebra action be a Poisson action.

Next, we ask when a map $\mu:M\to \mathfrak g ^*$, thought of as a section of $A^*$, is a Poisson map, where $A^*$ carries the Poisson structure $\widehat\Pi + \Pi_A$. The functions on  $A^*$  which are affine on fibres (which are enough for our purposes, since they suffice to test bivector fields)  are generated  by two kinds.  The first, which we will call ``vertical" (V), are the linear functions on fibres coming from fixed elements $a$ of $\frakg$.  The second are ``horizontal" (H); these are the pullbacks by the bundle projection of functions $f$ on $M.$  By abuse of notation, we will also denote the corresponding functions on $A^*$ by $a$ and $f$.  Thus, we can write the Poisson bracket relations on $A^*$ as
\begin{subequations}
\label{eq:sum is Poisson}
\begin{align}
\{f,g\}_{\widehat\Pi + \Pi_A}&=\{f,g\}_\Pi\\
\{a,f\}_{\widehat\Pi + \Pi_A}&=\rho a\cdot f\\
\{a,b\}_{\widehat\Pi + \Pi_A}&=[a,b].
\end{align}    
\end{subequations}
(In the equations above, the Poisson bracket on the right side is that on $M,$ while the one on the left is that on $A^*.$) 

To check whether a section $\mu$ is a Poisson map, it suffices to see whether each of the three bracket relations above is preserved under pullback by $\mu$. For the first (HH) one, this is true for any $\mu$, since the HH part of the Poisson structure on $A^*$ is essentially the structure on $M$.

For the second (VH) equation, we first note that the pulled back function $a\circ \mu$ can be written as the pairing $\langle a,\mu\rangle$ (where $\mu$ in the first expression is a section of $A^*$, while in the second it is a map from $M$  to $\mathfrak{g}^*$).  We thus have to check whether $\{\langle\mu,a \rangle,f\}_\Pi=\rho a \cdot f$ for all Lie algebra elements $a$ and all functions $f$ on $M$.  But that is precisely the condition that $\mu$ be a momentum map (or section).

Finally, the third (VV) equation is satisfied exactly when the momentum map $\mu$ is a Poisson map.

We conclude that the structure given by Equations~\eqref{eq:sum is Poisson} is a Poisson structure on $A^*$ if and only if the action of $\mathfrak g$ preserves the Poisson structure on $M$.  When this is satisfied, the section $\mu$ of $A^*$ is a Poisson map if and only if the action is hamiltonian with momentum map $\mu$.
\end{Example}

We go on from the example of action Lie algebroids to the general case. To begin, we translate the condition that two bivector fields commute with respect to the Schouten bracket into a relation between the corresponding brackets on scalar functions, defined for a bivector field $\Phi$ by
\begin{equation*}
  \{ f, g \}_\Phi := \Phi(df, dg).
\end{equation*}
The bracket is bilinear, antisymmetric, and a derivation in each argument. It satisfies the Jacobi identity if and only if $[\Phi, \Phi] = 0$. 

\begin{Lemma}
\label{lem:SchoutenCom}
Let $\Phi$ and $\Psi$ be bivector fields on $M$. Let 
\begin{equation*}
  C(f,g,h) 
  :=  \{f, \{g,h\}_\Psi \}_\Phi + \{ f, \{g, h\}_\Phi \}_\Psi 
  + \text{cycl.~perm.}
\end{equation*}
for $f, g, h \in C^\infty(M)$. The following are equivalent:
\begin{itemize}

\item[(i)] $[\Phi,\Psi] = 0$

\item[(ii)] $C(f,g,h) = 0$ for all functions $f,g,h$.

\end{itemize}
\end{Lemma}
\begin{proof}
It is straightforward to show, e.g.~in local coordinates, that that $C(f,g,h)$ is obtained by applying the 3-vector field $[\Phi,\Psi]$ to $f \otimes g \otimes h$, which implies the equivalence. It also follows that $C(f,g,h)$ is totally antisymmetric and a derivation in each argument.
\end{proof}

We now consider the special case where the underlying manifold is $A^*$, $\Phi = \widehat{\Pi}$ is the horizontal lift of a Poisson bivector on $M$, and $\Psi = \Pi_A$ is the Poisson bivector associated to the Lie algebroid structure on $A^*$. The bracket of $\widehat{\Pi}$ is given by
\begin{subequations}
\label{eqs:PiLiftBrackets}
\begin{align}
\label{bracketwidehatPi}
  \{f, g\}_{\widehat{\Pi}}
  &= \{f,g\}_\Pi
  \\
  \{a, f\}_{\widehat{\Pi}}
  &= D_{X_f} a
  \\
  \{a,b\}_{\widehat{\Pi}} 
  &= \Pi(Da, Db)
  \,,
\end{align}    
\end{subequations}
where $\{f,g\}_\Pi$ is the Poisson bracket on $M$, $X_f = \{\Empty, f\}$ is the hamiltonian vector field generated by $f$, and the notation $\Pi(Da, Db)$ denotes the tensor pairing $\langle Da \otimes Db, \Pi \rangle$. In local coordinates, we have
\begin{equation*}
  \Pi(Da, Db)
  = \Pi^{ij}(D_i a)(D_j b)
  \,,
\end{equation*}
where $D_i= D_{\frac{\partial}{\partial x^i}}$. The bracket of $\Pi_A$ is given by
\begin{subequations}
\label{eqs:PiABrackets}
\begin{align}
  \{f, g\}_{\Pi_A}
  &= 0
  \\
  \{a, f\}_{\Pi_A}
  &= \rho a \cdot f
  \\
  \{a,b\}_{\Pi_A} 
  &= [a, b]
  \,,
\end{align}    
\end{subequations}
where $[~,~]$ is the bracket of $A$ and $\rho$ is the anchor, as usual. For an action Lie algebroid and constant sections $a$, $b$ we retrieve Equations~\eqref{eq:sum is Poisson}. The following is the first of the two main results of this section.

\begin{Theorem}
\label{thm:compatibilityonA*}
Let $A \to M$ be a Lie algebroid equipped with a connection $D$. Let $\Pi$ be a bivector field on $M$. Let $\widehat{\Pi} \in \frakX^2(A^*)$ be the horizontal lift of $\Pi$ and $\Pi_A \in \frakX^2(A^*)$ the Lie algebroid Poisson bivector field. The following are equivalent:
\begin{enumerate}

\item[(i)] The bivector fields commute,
\begin{equation*}
  [\widehat{\Pi}, \Pi_A] = 0
  \,,
\end{equation*}
with respect to the Schouten bracket.

\item[(ii)] The conditions
\begin{align*}
  \chD_a \Pi &= 0  
  \\
  (D_v T_A)(a,b) - R(v, \rho a)b + R(v, \rho b)a
  &= 0
\end{align*}
hold for all $a, b \in A$ and all $v \in \Pish(T^* M)$, where $T_A$ is the $A$-torsion \eqref{eq:Atorsion} and $R(v,w)a = D_vD_w a - D_wD_v a - D_{[v,w]}a$ the curvature.
    
\end{enumerate}
\end{Theorem}

\begin{Remark}
The first condition in (ii) above is that the connection $D$ be Poisson anchored.  The second is a compatibility condition between the Lie algebroid structure, the connection, and the symplectic leaves of the Poisson structure.  Note that  $T_A$ involves the Lie algebroid bracket as well as the anchor.  It would be interesting to better understand this condition.

In the case of an action Lie algebroid with the trivial connection, the second condition always holds.  Not only is the curvature zero, but the covariant derivative of the $A$-torsion vanishes because the algebroid bracket of constant sections is again constant.  
\end{Remark}

\begin{proof}[Proof of Theorem \ref{thm:compatibilityonA*}]
We define
\begin{equation*}
  C(F,G,H) :=
  \{F, \{G,H\}_{\Pi_A} \}_{\widehat{\Pi}} +
  \{F, \{G,H\}_{\widehat{\Pi}} \}_{\Pi_A}
  + \text{cycl.~perm.}
\end{equation*}
for all $F,G,H \in C^\infty(A^*)$. As noted in the proof of Lemma~\ref{lem:SchoutenCom}, $C(F,G,H)$ is totally antisymmetric and a derivation in each argument. 

The condition $C(F,G,H) = 0$ holds if and only if it holds for $F, G, H$ each being either a function $f,g,h$ on $M$ or a section $a,b,c$ of $A$. For three functions, we have
\begin{equation}
\label{eq:PoissComm2}
  C(f,g,h) = 0
  \,,
\end{equation}
since $\{f,g\}_{\Pi_A}= 0$. 
 
For two functions on $M$ and one section of $A$ we have:
\begin{equation}
\label{eq:PoissComm3}
\begin{split}
  C(f,g,a)
  &=
    \{f, \{g, a\}_{\Pi_A}\}_{\widehat{\Pi}}
  + \{g, \{a, f\}_{\Pi_A}\}_{\widehat{\Pi}}
  + \{a, \{f, g\}_{\Pi_A}\}_{\widehat{\Pi}}
  \\
  &{}\quad
  + \{f, \{g, a\}_{\widehat{\Pi}}\}_{\Pi_A}
  + \{g, \{a, f\}_{\widehat{\Pi}}\}_{\Pi_A}
  + \{a, \{f, g\}_{\widehat{\Pi}}\}_{\Pi_A}
  \\
  &= 
  - \{f, \rho a \cdot g \}_\Pi + \{g, \rho a \cdot f \}_\Pi
  \\
  &{}\quad
  + \rho(D_{X_g} a) \cdot f
  - \rho(D_{X_f} a) \cdot g
  + \rho a \cdot \{f,g\}_\Pi
  \\
  &=
  - \Pi(d f, d \Lie_{\rho a} g)
  + \Pi(d g, d \Lie_{\rho a} f)
  \\
  &{}\quad
  + \bigl\langle df, \rho(D_{X_g} a) \bigr\rangle
  - \bigl\langle dg, \rho(D_{X_f} a) \bigr\rangle
  + \rho a \cdot \Pi(df,dg)
  \\
  &= 
  - \langle \Lie_{\rho a} dg , \Pish df \rangle
  + \bigl\langle dg, \rho(D_{\Pish df} a) \bigr\rangle  
  \\
  &{}\quad
  + \langle \Lie_{\rho a} df , \Pish dg \rangle
  - \bigl\langle df, \rho(D_{\Pish dg} a) \bigr\rangle  
  + \rho a \cdot \Pi(df,dg)
  \,.
\end{split}
\end{equation}
For all sections $a$ of $A$, 1-forms $\beta$ on $M$, and vector fields $v$ on $M$, we have
\begin{equation}
\label{eq:PoissComm4}
\begin{split}
  \langle \chD_a \beta, v \rangle
  &=
  \rho a \cdot \langle \beta, v \rangle
  - \langle \beta, \chD_a v \rangle
  \\
  &=
  \langle \Lie_{\rho a} \beta, v \rangle
  + \langle \beta, \Lie_{\rho a} v \rangle
  - \langle \beta, \chD_a v \rangle
  \\
  &=
  \langle \Lie_{\rho a} \beta, v \rangle
  - \bigl\langle \beta, \rho( D_{v} a ) \bigr\rangle
  \,.
\end{split}
\end{equation}
Using this relation, Equation~\eqref{eq:PoissComm3} can be written as
\begin{equation*}
\begin{split}
  C(f,g,a)
  &= 
  - \langle \chD_a dg, \Pish df \rangle
  + \langle \chD_a df, \Pish dg \rangle
  + \rho a \cdot \Pi(df,dg)
  \\
  &=  (\chD_a \Pi)(df, dg)
  \,.
\end{split}
\end{equation*}
We conclude that the condition  $C(f,g,a) = 0$ holds for all $f$, $g$, and $a$, if and only if 
\begin{equation}
\label{eq:PoissComm4b}
  \chD \Pi = 0
  \,;
\end{equation}
i.e., if and only if $A$ is Poisson anchored with respect to $D$.

For one function on $M$ and two sections of $A$, we have:
\begin{equation}
\label{eq:PoissComm5}
\begin{split}
  C(f,a,b)
  &=
    \{f, \{a, b\}_{\Pi_A}\}_{\widehat{\Pi}}
  + \{a, \{b, f\}_{\Pi_A}\}_{\widehat{\Pi}}
  + \{b, \{f, a\}_{\Pi_A}\}_{\widehat{\Pi}}
  \\
  &{}\quad
  + \{f, \{a, b\}_{\widehat{\Pi}}\}_{\Pi_A}
  + \{a, \{b, f\}_{\widehat{\Pi}}\}_{\Pi_A}
  + \{b, \{f, a\}_{\widehat{\Pi}}\}_{\Pi_A}
  \\
  &= 
  -D_{X_f} [a,b] 
  + D_{X_{\rho b \cdot f}} a 
  - D_{X_{\rho a \cdot f}} b
  \\
  {}&\quad
  + \{f, \Pi^{ij}(D_i a)(D_j b) \}_{\Pi_A}
  + [a, D_{X_f} b] - [b, D_{X_f}a]
  \,.
\end{split}
\end{equation}
Using the derivation property of the brackets and that $\{ f, \Pi^{ij} \}_{\Pi_A} = 0$, we obtain
\begin{equation*}
\begin{split}
  \{f, \Pi^{ij}(D_i a)(D_j b) \}_{\Pi_A}
  &=
    \Pi^{ij} \{f, (D_i a)\}_{\Pi_A} (D_j b)
  + \Pi^{ij} (D_i a) \{f, (D_j b)\}_{\Pi_A}
  \\
  &=
  - \Pi^{ij} \bigl( \rho(D_i a) \cdot f \bigr) (D_j b)
  - \Pi^{ij} (D_i a) \bigl( \rho(D_j b) \cdot f \bigr)
\end{split}    
\end{equation*}
From Equation~\eqref{eq:PoissComm4}, we get
\begin{equation*}
  \rho(D_i a) \cdot f
  = \partial_i \cdot (\rho a  \cdot f) 
  - \langle \chD_a df, \partial_i \rangle
  \,,
\end{equation*}
so that
\begin{equation*}
  - \Pi^{ij} \bigl( \rho(D_i a) \cdot f \bigr) (D_j b)
  = D_{X_{\rho a \cdot f}} b 
  + \Pi(\chD_a df, Db)
  \,.
\end{equation*}
Inserting this in Equation~\eqref{eq:PoissComm5}, we obtain
\begin{equation}
\label{eq:PoissComm5b}
\begin{split}
  C(f,a,b)
  &=   -D_{X_f} [a,b] 
  + [D_{X_f}a, b] + [a, D_{X_f}b]
  \\
  &{}\quad
  + \Pi(\chD_a df, Db)
  + \Pi(Da, \chD_b df )
  \,.
\end{split}    
\end{equation}
In order to rewrite the second line, we will use the relation
\begin{equation}
\label{eq:PoissComm5c}
\begin{split}
  \Pi(\chD_a \beta, \gamma)
  &= \rho a \cdot \Pi(\beta, \gamma) - \Pi(\beta, \chD_a \gamma) 
  - (\chD_a \Pi)(\beta, \gamma)
  \\
  &= \rho a \cdot \langle \gamma, \Pish \beta \rangle 
  - \langle \chD_a \gamma, \Pish \beta \rangle 
  - (\chD_a \Pi)(\beta, \gamma)
  \\
  &= \langle \gamma, \chD_a \Pish \beta \rangle 
  - (\chD_a \Pi)(\beta, \gamma)
  \,,
\end{split}    
\end{equation}
which holds for all sections $a$ of $A$ and 1-forms $\beta$, $\gamma$ on $M$. Using this relation, Equation~\eqref{eq:PoissComm5b} can be written as
\begin{equation}
\label{eq:PoissComm6}
\begin{split}
  C(f,a,b)
  &=   
  -D_{X_f} [a,b] 
  + [D_{X_f}a, b] + [a, D_{X_f}b]
  - D_{\chD_a X_f} b + D_{\chD_b X_f} a
  \\
  &{}\quad
  - (\chD_a \Pi)(df, Db)
  + (\chD_b \Pi)(df, Da)     
  \,.
\end{split}    
\end{equation}
Let $T_A$ be the $A$-torsion of $D$, $R$ the (usual) curvature of $D$, and $v$ a vector field on $M$. Then
\begin{equation}
\label{eq:PoissComm6b}
\begin{split}
  (D_v T_A)(a,b)
  &= D_v\bigl( D_{\rho a} b - D_{\rho b} a - [a,b] \bigr)
  \\
  &{}\quad
  - D_{\rho( D_v a)} b + D_{\rho b} D_v a + [D_v a, b]
  \\
  &{}\quad
  - D_{\rho a} D_v b + D_{\rho( D_v b)} a + [a, D_v b]
  \\
  &= - D_v[a,b] + [D_v a, b] + [a, D_v b]
  \\
  &{}\quad
  + R(v, \rho a)b + D_{[v, \rho a]} b - D_{\rho(D_v a)} b
  \\
  &{}\quad  
  - R(v, \rho b)a - D_{[v, \rho b]} a + D_{\rho(D_v b)} a
  \\
  &= 
  - D_v[a,b] + [D_v a, b] + [a, D_v b]
  - D_{\chD_a v} b + D_{\chD_b v} a 
  \\
  &{}\quad
  + R(v, \rho a)b - R(v, \rho b)a
  \,.
\end{split}    
\end{equation}
Comparing this with Equation~\eqref{eq:PoissComm6}, we get
\begin{equation}
\label{eq:PoissComm7}
\begin{split}
  C(f,a,b)
  &= (D_{X_f} T_A)(a,b) - R(X_f, \rho a)b + R(X_f, \rho b)a
  \\
  &{}\quad
  - (\chD_a \Pi)(df, Db)
  + (\chD_b \Pi)(df, Da)     
  \,.
\end{split}    
\end{equation}
We conclude that, under the assumption that $\chD\Pi = 0$, the condition $C(f,a,b) = 0$ holds for all $f$, $a$, and $b$ if and only if 
\begin{equation*}
  (D_{X_f} T_A)(a,b) - R(X_f, \rho a)b + R(X_f, \rho b)a
  = 0
\end{equation*}
holds. Since the hamiltonian vector fields $X_f$ span the distribution $\Pish(T^* M)$, this is the second of conditions~(ii) of the theorem.

Finally, for three sections of $A$, we have
\begin{equation}
\label{eq:PoissComm10}
\begin{split}
  C(a,b,c)
  &=
    \{a, \{b, c\}_{\Pi_A}\}_{\widehat{\Pi}}
  + \{a, \{b, c\}_{\widehat{\Pi}}\}_{\Pi_A}
  + \text{c.p.}
  \\
  &= 
   \Pi\bigl(Da, D[b,c]\bigr) 
  + \{a, \Pi(Db, Dc)\}_{\Pi_A}
  + \text{c.p.}
  \\
  &= 
    \Pi^{ij}(D_i c)\bigl( D_j [a,b]\bigr) 
  + \{a, \Pi^{ij}(D_i b)(D_j c)\}_{\Pi_A}
  + \text{c.p.}
  \\
  &= 
  - (D_i c) \bigl(D_{X_{x^i}} [a,b]\bigr) 
  + \{a, \Pi^{ij}(D_i b)(D_j c)\}_{\Pi_A}
  + \text{c.p.}
  \,, 
\end{split}
\end{equation}
where we have used in the third line that we can permute $a$,$b$, and $c$ cyclically in the first term. In the last step, we have used that the hamiltonian vector field is defined by $X_f = \{\Empty, f\} = - \Pish df$, so that $X_{x^i} = - \Pi^{ij}\partial_j$ and $\Pi^{ij}D_j = - D_{X_{x^i}}$. Using the derivation property of the Poisson brackets, we get for the second term
\begin{equation}
\label{eq:PoissComm11}
\begin{split}
  \{a, \Pi(Db, Dc)\}_{\Pi_A}
  &= 
  \{a, \Pi^{ij} (D_i b)(D_j c)\}_{\Pi_A}
  \\
  &= 
  \{a, \Pi^{ij} (D_i b)\}_{\Pi_A} (D_j c)
  + (D_i b) \{a, \Pi^{ij} (D_j c) \}_{\Pi_A}
  \\
  &{}\quad
  - \{a, \Pi^{ij}\}_{\Pi_A}(D_i b)(D_j c)
  \\
  &= 
  [a, D_{X_{x^j}} b] (D_j c)
  - (D_i b) [a, D_{X_{x^i}} c]
  - (\rho a \cdot \Pi^{ij})(D_i b)(D_j c)
  \,.  
\end{split}
\end{equation}
For the last term on the right side of Equation~\eqref{eq:PoissComm11} we have the relation
\begin{equation*}
\begin{split}
  \rho a \cdot \Pi^{ij}
  &=
  \rho a \cdot \Pi(dx^i, dx^j)
  \\
  &=
  (\chD_a \Pi)(dx^i, dx^j)
  + \Pi(\chD_a dx^i, dx^j)
  + \Pi(dx^i, \chD_a dx^j)
  \\
  &=
    \langle dx^j,  \chD_a \Pish dx^i \rangle  
  - \langle dx^i,  \chD_a \Pish dx^j \rangle  
  - (\chD_a\Pi)(dx^i, dx^j)
  \\
  &=
  - \langle dx^j, \chD_a X_{x^i} \rangle  
  + \langle dx^i, \chD_a X_{x^j} \rangle  
  - (\chD_a\Pi)(dx^i, dx^j)
  \,,
\end{split}    
\end{equation*}
where we have used Equation~\eqref{eq:PoissComm5c}. With this, the last term on the right side of Equation~\eqref{eq:PoissComm11} takes the form
\begin{equation}
\label{eq:PoissComm11b}
  (\rho a \cdot \Pi^{ij})(D_ib)(D_j c)
  =
  - (D_{\chD_a X_{x^i}} c)(D_i b)  
  + (D_{\chD_a X_{x^i}} b)(D_i c)
  - (\chD_a \Pi)(Db, Dc)
  \,.
\end{equation}
By inserting first~\eqref{eq:PoissComm11} and then~\eqref{eq:PoissComm11b} into \eqref{eq:PoissComm10}, we obtain
\begin{equation*}
\begin{split}
  C(a,b,c)
  &=
  - (D_i c)\bigl( D_{X_{x^i}} [a, b]\bigr) 
  + [a, D_{X_{x^j}} b] (D_j c)
  - (D_i b) [a, D_{X_{x^i}} c]
  \\
  &{}\quad
  + (D_{\chD_b X_{x^i}} a)(D_i c)  
  - (D_{\chD_a X_{x^i}} b)(D_i c)
  \\
  &{}\quad  
  + (\chD_a \Pi)(Db, Dc)
  + \text{c.p.}
  \\
  &= 
  \bigl(- D_{X_{x^i}} [a, b]
      + [D_{X_{x^i}}a, b] 
      + [a, D_{X_{x^i}} b]
      - D_{\chD_a X_{x^i}} b  
      + D_{\chD_b X_{x^i}} a
  \bigr)(D_i c)
  \\
  &{}\quad  
  + (\chD_a \Pi)(Db, Dc)
  + \text{c.p.}
  \\
  &=
  \bigl( (D_{X_{x^i}} T_A)(a,b) 
  - R(X_{x^i}, \rho a)b + R(X_{x^i}, \rho b)a 
  \bigr)(D_i c)
  \\
  &{}\quad
  + (\chD_a \Pi)(Db, Dc)
  + \text{c.p.}  
  \,,
\end{split}
\end{equation*}
where in the last step we have used Equation~\eqref{eq:PoissComm6b} for $v = X_{x^i}$. We conclude that $C(a,b,c) = 0$ does not yield a new condition, and our proof is complete.
\end{proof}

Our second result in this section concerns sections $\mu$ of $A^*$ as candidates for (bracket compatible) momentum sections with respect to a connection which is Poisson anchored. We will equip $A^*$ with the bivector field $\widehat{\Pi}+{\Pi_A},$ which, as we noted earlier, is generally not a Poisson structure if the connection is not flat.  We may then ask when $\mu$ is a ``Poisson" map, i.e. when it pulls back brackets on functions for $\widehat{\Pi}+{\Pi_A}$ on $A^*$ to those for $\Pi$ on $M$. 

\begin{Definition}
\label{def:BivectorMap}
Let $(M,\Phi)$ and $(N,\Psi)$ be pairs consisting of a manifold and a bivector field on the manifold. A smooth map $\phi: M \to N$ will be called a \textbf{bivector map} if
\begin{equation*}
  \phi^* \{f,g\}_\Psi = \{\phi^* f, \phi^* g \}_\Phi
\end{equation*}
for all $f,g \in C^\infty(N)$.
\end{Definition}

\begin{Theorem}
\label{thm:momentumsectionasPoissonmap}
Let $\Pi$ be a Poisson bivector field on $M$. Let $A \to M$ be a Lie algebroid equipped with a Poisson anchored connection $D$. Let $\hat{\Pi}$ be the horizontal lift of $\Pi$ to $A^*$ and $\Pi_A$ the Lie algebroid Poisson bivector field on $A^*$. A section $\mu$ of $A^*$ is a bracket compatible momentum section if and only if it is a bivector map from $(M,\Pi)$ to $(A^*,\widehat{\Pi}+\Pi_A)$.
\end{Theorem}

\begin{proof}
As in the proof of Theorem~\ref{thm:compatibilityonA*}, we can consider separately those functions on $A^*$ that come from functions $f,g$ on $M$ and those that come from sections $a,b$ of $A$. By adding Equations~\eqref{eqs:PiLiftBrackets} and Equations~\eqref{eqs:PiABrackets}, we obtain
\begin{subequations}
\begin{align}
\label{eq:muBrackets01}
  \{f, g\}_{\widehat{\Pi}+\Pi_A}
  &= \{f,g\}_\Pi
  \\
\label{eq:muBrackets02}
  \{a, f\}_{\widehat{\Pi}+\Pi_A}
  &= D_{X_f} a +\rho a \cdot f
  \\
\label{eq:muBrackets03}
  \{a,b\}_{\widehat{\Pi}+\Pi_A} 
  &= \Pi(Da, Db) + [a,b]
  \,.
\end{align}    
\end{subequations}
Equation~\eqref{eq:muBrackets01} shows that for $f$ and $g$ there are no conditions on $\mu$ to be a bivector map, just as in the case of an action Lie algebroid with the trivial connection.

For $a$ and $f$, we deduce from Equation~\eqref{eq:muBrackets02} that for $\mu$ to be a bivector map,
\begin{equation}
\label{eq:muBrackets04}
  (D_{X_f} a )\circ\mu + \rho a \cdot f = \{ \langle\mu,a\rangle,f\}_\Pi
\end{equation}
must be satisfied. This equation can be rewritten as
\begin{equation*}
\begin{split}
  \langle \rho a, df \rangle
  &=
  \rho a \cdot f
  \\
  &= 
  \{ \langle\mu,a\rangle,f\}_\Pi - (D_{X_f} a )\circ\mu
  \\
  &= X_f \cdot \langle\mu, a \rangle - \langle\mu, D_{X_f} a \rangle
  \\
  &= \langle D_{X_f} \mu, a \rangle
  \\
  &= \Pi\bigl( \langle D\mu, a\rangle, df \bigr)
  \,,
\end{split}
\end{equation*}
where we have used the Definition~\eqref{eq:DPairing} of the dual connection. Since this relation must hold for all functions $f$, we see that~\eqref{eq:muBrackets04} is equivalent to condition $(\mathrm{H2_{Poi}})$ of Definition~\ref{def:HamLAPoisson}, the condition for $\mu$ to be a momentum section.

Finally, for $a$ and $b$, we deduce from Equation~\eqref{eq:muBrackets03} that for $\mu$ to be a bivector map, we must have 
\begin{equation}
\label{eq:muBrackets05}
  \bigl( \Pi(Da, Db) + [a,b]\bigr)\circ \mu 
  = \{\langle \mu,a \rangle,\langle \mu,b  \rangle\}_\Pi
  \,.
\end{equation}
The left side of Equation~\eqref{eq:muBrackets05} can be written as
\begin{equation*}
  \bigl( \Pi(Da, Db) + [a,b]\bigr)\circ \mu 
  = \Pi\bigl( \langle\mu, Da\rangle, \langle\mu, Db\rangle \bigr)
  + \langle\mu, [a,b] \rangle
  \,.
\end{equation*}
Assuming that $(\mathrm{H2_{Poi}})$ holds, the right side of Equation~\eqref{eq:muBrackets05} can be written as
\begin{equation*}
\begin{split}
  \{\langle \mu,a \rangle,\langle \mu,b  \rangle\}_\Pi
  &=
  \Pi\bigl( d\langle \mu,a \rangle, d\langle \mu,b\rangle \bigr)
  \\
  &=
  \Pi\bigl( 
    \langle D\mu, a \rangle + \langle \mu, Da \rangle, 
    \langle D\mu, b \rangle + \langle \mu, Db \rangle \bigr)
  \\
  &=
    \Pi\bigl(\langle D\mu, a \rangle, \langle D\mu, b \rangle \bigr)
  + \langle \mu, D_{\rho a} b \rangle
  - \langle \mu, D_{\rho b} a \rangle
  \\
  &\quad{}
  + \Pi\bigl(\langle \mu, Da \rangle, \langle \mu, Db \rangle \bigr)
  \,.
\end{split}
\end{equation*}
Using the last two equations, Equation~\eqref{eq:muBrackets05} takes the form
\begin{equation*}
  0 = \Pi\bigl(\langle D\mu, a \rangle, \langle D\mu, b \rangle \bigr)
  + \langle\mu, \TorA(a,b) \rangle
  \,,
\end{equation*}
which is the torsion form~\eqref{eq:H3PoissonTorsion} of $(\mathrm{H2_{Poi}})$.
\end{proof}

\bibliographystyle{halpha}
\bibliography{PoissHamLAs}

\end{document}